\documentclass[10pt]{amsart}

\usepackage{amsmath,comment}
\usepackage{amssymb}
\usepackage{enumerate}
\usepackage{amsthm}
\usepackage{tikz-cd}
\usepackage{tikz}
\usetikzlibrary{arrows.meta,arrows}
\usepackage{hyperref}%[colorlinks]{hyperref}

\newcommand{\Z}{\mathbb{Z}}

\newcommand{\Q}{\mathbb{Q}}

\newcommand{\F}{\mathbb{F}}

\newcommand{\Aa}{\mathbb{A}}

\newcommand{\val}{\operatorname{val}}
\newcommand{\Aut}{\mathsf{Aut}}
\newcommand{\red}{\mathrm{red}}
\newcommand{\Stab}{\mathsf{Stab}}
\newcommand{\dist}{\mathsf{dist}}

\newcommand{\SA}{\textbf{(SA)}}
\newcommand{\XD}{\textbf{(XD)}}

\numberwithin{equation}{section}

\newtheorem{lemma}{Lemma}
\newtheorem{corollary}[lemma]{Corollary}
\newtheorem{theorem}[lemma]{Theorem}
\newtheorem*{extTheorem}{Theorem}
\newtheorem{proposition}[lemma]{Proposition}

\theoremstyle{definition}
\newtheorem{definition}[lemma]{Definition}

\newtheorem{question}[lemma]{Question}
\newtheorem*{remark}{Remark}
\newif\ifIMFT
\IMFTfalse

\title[Residual Transitivity implies Minimality]{Residual Transitivity implies Minimality for Markoff Surfaces over $p$-adic Integers, by Means of $p$-adic Flows}
\author{Seung uk Jang}
\date{January 29, 2025. Last revised \today}

\begin{document}
\allowdisplaybreaks

\begin{abstract}
    Let $X_D^\ast$ be the non-singuar locus of the Markoff surface $X_D\colon x^2+y^2+z^2=xyz+D$ and consider the set of its $p$-adic integer points $X_D^\ast(\mathbb{Z}_p)$. It is known to Bourgain, Gamburd, and Sarnak that the modulo $p$ transitivity by algebraic automorphisms of $X_0^\ast$ implies minimality of $X_0^\ast(\mathbb{Z}_p)$ by algebraic automorphisms. In this paper, we provide an alternative proof of this fact, by some techniques to study $p$-adic analytic flows. This establish a slight generalization to those parameters $D$ congruent to $0$ modulo $p^2$ or $(D-4)$ being a nonzero quadratic residue.
\end{abstract}
\maketitle
%\setcounter{tocdepth}{1} % part,chapters,sections
%\tableofcontents

\section{Introduction}

Let $X_D^\ast$ be the algebraic surface defined over a local ring $R$ by
\[X_D^\ast\colon\begin{cases} x^2+y^2+z^2=xyz+D; \\ \text{one of }2x-yz,2y-xz,2z-xy\text{ is invertible}. \end{cases}\]
Here, $D\in R$ is the parameter. If the defining ring $R$ is a field, this means that $X_D^\ast$ is the locus of nonsingular points of the surface $X_D\colon x^2+y^2+z^2=xyz+D$ in the affine 3-space $\Aa^3_{xyz}$.

The group $\mathsf{Aut}(X_D^\ast)$ of algebraic automorphisms of $X_D^\ast$ is generated by sign-change automorphisms $(x,y,z)\mapsto(x,-y,-z)$, etc., permutations of coordinates, and the so-called Vieta involutions
\begin{align*}
    s_x\begin{pmatrix} x \\ y \\ z \end{pmatrix} &= \begin{pmatrix} yz-x \\ y \\ z \end{pmatrix}, & s_y\begin{pmatrix} x \\ y \\ z \end{pmatrix} &= \begin{pmatrix} x \\ xz-y \\ z \end{pmatrix}, & s_z\begin{pmatrix} x \\ y \\ z \end{pmatrix} &= \begin{pmatrix} x \\ y \\ xy-z \end{pmatrix}.
\end{align*}
(See \cite[Thm. 2]{el-huti} and \cite[\S{2.2}]{Goldman2003}.) The subgroup generated by Vieta involutions will be denoted by $\Gamma=\langle s_x,s_y,s_z\rangle$ throughout, which is an index 24 normal subgroup of the automorphism group.

The action of $\Aut(X_D^\ast)$ on $X_D^\ast$ over finite prime fields had been an interest in various works, including \cite{BGS16,BGS16details,Chen2024,EFLMT2023,Martin2025}. One of the key result in the theory is that the group $\Aut(X_0^\ast)$ acts on $X_0^\ast(\Z/p\Z)$ transitively for $p>3.489\times 10^{392}$. 
It is natural to ask whether the same holds for higher powers of $p$. A. Gamburd has announced the following result in \cite[\S{1.5}, Thm. 4]{Gamburd2022}. Here, a group action on a topological space is \emph{minimal} if every orbit is dense.
\begin{extTheorem}[Bourgain--Gamburd--Sarnak]
    Suppose $p>3$ is a prime such that $\Aut(X_0^\ast)$ acts transitively on $X_0^\ast(\Z/p\Z)$. Then the group $\Aut(X_0^\ast)$ acts minimally on $X_0^\ast(\Z_p)$.
\end{extTheorem}

In this article, we prove a slight generalization of the fact, using a technique independent to what was mentioned in \cite{Gamburd2022}. Recall that we can extend modulo $p$ Legendre symbols $\left(\frac{{}\cdot{}}{p}\right)$ to $p$-adic integers.

\begin{theorem}
    \label{thm:main}
    Suppose $p>3$ is a prime and $D\in\Z_p$ be such that $D\equiv 0\pmod{p^2}$ or $\left(\frac{D-4}{p}\right)=1$. 
    If $\Aut(X_D^\ast)$ acts transitively on $X_D^\ast(\Z/p\Z)$, then the group acts minimally on $X_D^\ast(\Z_p)$.
\end{theorem}

%The case $p=5$ and $D\equiv 3\pmod{p}$ is one of the exceptional cases when residual transitivey does not imply minimality (see section \ref{sec:generalization} for details).

\subsection{Proof Sketch} The proof is based on techniques to study $p$-adic algebraic dynamics introduced in works like \cite{CantatXie2018,CJ24}.

First, we decompose $X_D^\ast(\Z_p)$ into fibers of the reduction map $\red\colon X_D^\ast(\Z_p)\to X_D^\ast(\Z/p\Z)$, each being an analytic copy of $\Z_p^2$. Thanks to the modulo $p$ transitivity assumption, fibers of this reduction map are permuted transitively by $\Gamma$. 
Hence it suffices to show that \emph{one} of the fiber $U$ is minimal by its stabilizer $\Stab_\Gamma(U)$ in $\Gamma=\langle s_x,s_y,s_z\rangle$.

We find $U$ with its analytic parametrization $\Psi\colon\Z_p^2\to U$ that validates the following analysis (see Lemma \ref{lem:minimality-criterion}).
\begin{enumerate}
    \item First, we show that there is a subgroup of $\Stab_\Gamma(U)$ which acts minimally on a subdisk $\Psi((u_0,v_0)+(p\Z_p)^2)\subset U$.
    \item Next, we show that $\Stab_\Gamma(U)$, or more precisely $\Psi^{-1}\Stab_\Gamma(U)\Psi$, acts transitively on the reduction $(\Z_p/p\Z_p)^2$ of $\Z_p^2$.
\end{enumerate}
Step (1) above is precisely where a technique of $p$-adic flows comes into play (see Propositions \ref{lem:local-minimality} and \ref{lem:conjugate-local-minimality}).

Such analysis is done under certain assumptions on points of $U$ (see Proposition \ref{lem:minimal-polydisk-eg}). If $\left(\frac{D-4}{p}\right)=1$, then we can choose a specific disk $U$ (see Lemma \ref{lem:special-point} and Proposition \ref{lem:minimal-polydisk-exceptional}) and show that it is a minimal polydisk. If $\left(\frac{D-4}{p}\right)\neq 1$ yet $D\equiv 0\pmod{p^2}$, then $p\equiv 3\pmod{4}$ must hold and we incorporate D. Martin's approach \cite{Martin2025} to W. Chen's result \cite[Thm. 1.2.5]{Chen2024} to have modulo $p^2$ divisibility of $\Gamma$-orbits. Together with a counting argument, we show that there still exists a minimal polydisk.

\subsection{Outline of the Paper} Sections \ref{sec:bell-poonen-theory} and \ref{sec:computation-gadgets} lists some preliminary tools to describe the analysis. Using these tools, section \ref{sec:local-analysis-stabilizers} establishes local analysis of known stabilizer elements. Section \ref{sec:counting-argument} establishes a ``small move'' of some points by $\Gamma$. Section \ref{sec:criteria-of-good-polydisk} studies which fiber of the reduction map $X_D^\ast(\Z_p)\to X_D^\ast(\Z_p/p\Z_p)$ is minimal. By the same token, we prove our main Theorem \ref{thm:main} in \S{\ref{sec:proof}}. 
Section \ref{sec:generalization} discusses on which parameters we can generalize this technique, with some questions that originate from this interest.

\subsection{Standing Notations} Throughout this article, $p\geq 3$ is an odd prime, $\Gamma=\langle s_x,s_y,s_z\rangle$ is the group generated by Vieta involutions, and $T_p(x)$ is the monic Chebyshev polynomial of degree $p$ (see Definition \ref{def:chebyshev} below).

\section{Bell--Poonen Flow and its Modulo $p^2$ Reduction}
\label{sec:bell-poonen-theory}

Given an analytic map $f\colon\Z_p^m\to\Z_p^m$, if it is congruent to the identity modulo $p$, then it is known that iterates of $f$ interpolate to a $p$-adic flow map $\Phi_f\colon\Z_p\times\Z_p^m\to\Z_p^m$ \cite{Bell2006,Bell2006:corrigendum,Poonen2014}. In this section, we study modulo $p^2$ reduction of this flow map, together with various corollaries that arise from this reduction.

\subsection{Definitions, Terms, and Notations}

We start by terms and notations to describe the necessary analysis.

On the field $\Q_p$ of $p$-adic numbers, we denote the $p$-adic valuation by $\val_p$ and the $p$-adic absolute value by $|{}\cdot{}|_p$. This absolute value is normalized so that $|p|_p=p^{-1}$. The ring of $p$-adic integers $\Z_p$ is the valuation ring by $\val_p$, i.e., the ring of $p$-adic numbers of absolute value at most 1. Its residue field is the field of order $p$, denoted by $\F_p$, and we have natural isomorphisms $\F_p\cong\Z_p/p\Z_p\cong\Z/p\Z$.

On the polydisk $\Z_p^m$, we denote by $x=(x_1,\ldots,x_m)$ its coordinates. Let $x^I=x_1^{i_1}\cdots x_m^{i_m}$ be the monomial associated to the multi-index $I=(i_1,\ldots,i_m)\in\Z_{\geq 0}^m$. By an \emph{analytic function} $f\colon\Z_p^m\to\Z_p$, we mean a function of the form $f(x)=\sum_I c_Ix^I$, where the coefficients $c_I$ are in $\Z_p$ yet $|c_I|_p\to 0$ as $I\to\infty$ (i.e., as the length $|I|:=i_1+\cdots+i_m$ goes to infinity). The set of analytic functions is denoted by $\Z_p\langle x\rangle$, which is completely normed by the \emph{Gauss norm}: if $f=\sum c_Ix^I\in\Z_p\langle x\rangle$, then its Gauss norm is defined as $\|f\|=\max_I|c_I|_p$.

By an \emph{analytic map} $f\colon\Z_p^m\to\Z_p^n$ we mean a vector of $n$ analytic functions, i.e., $f\in\Z_p\langle x\rangle^n$. The Gauss norm of an analytic map $f=(f_1,\ldots,f_n)$ is the maximum Gauss norm of its components: $\|f\|=\max\|f_i\|$. The \emph{identity map} $\mathrm{id}\colon\Z_p^m\to\Z_p^m$ is the distinguished element $(x_1,\ldots,x_m)\in\Z_p\langle x\rangle^m$.

We say two analytic maps or functions $f,g$ are congruent modulo $p^k$, denoted by $f\equiv g\pmod{p^k}$ or $f(x)\equiv g(x)\pmod{p^k}$, if the Gauss norm of their difference $\|f-g\|$ is $\leq p^{-k}$. In other words, if we reduce coefficients of $f$ and $g$ modulo $p^k$, they yield the same series (polynomials, in fact). For the notation $f(x)\equiv g(x)\pmod{p^k}$, if there is a potential confusion, we may add ``as functions of $x$'' to indicate that the congruence is about the Gauss norm.

Suppose $x_0,x_1\in\Z_p^m$ are two points that are congruent modulo $p^k$. Then the monomials $x_0^I$ and $x_1^I$ are congruent modulo $p^k$. Hence for any analytic function $f\in\Z_p\langle x\rangle$, the function values $f(x_0)$ and $f(x_1)$ are congruent modulo $p^k$. That is, $x_0\equiv x_1\pmod{p^k}$ implies $f(x_0)\equiv f(x_1)\pmod{p^k}$. This symbolism generalizes to an extent: if $x_0$, $x_1$ are maps $\in\Z_p\langle x\rangle^n$ such that $x_0\equiv x_1\pmod{p^k}$ then we have $f\circ x_0\equiv f\circ x_1\pmod{p^k}$ for any analytic function $f\colon\Z_p^n\to\Z_p$.

A continuous $\Z_p$-linear operator $T\colon\Z_p\langle x\rangle\to\Z_p\langle x\rangle$ has the \emph{operator norm} $\|T\|_{op}$ defined as the maximal Gauss norm on its image, i.e., $\sup_{f\in\Z_p\langle x\rangle}\|T(f)\|$. Typically, we deal with the post-composition operator: given an analytic self-map $f\colon\Z_p^m\to\Z_p^m$, we induces an operator $T_f\colon\Z_p\langle x\rangle\to\Z_p\langle x\rangle$, $\phi\mapsto\phi\circ f$. A remark of the above paragraph then shows that $\|T_f-T_g\|_{op}=\|f-g\|$; in particular, if $g=\mathrm{id}$, then $\|T_f-\mathrm{Id}\|_{op}=\|f-\mathrm{id}\|$, where $\mathrm{Id}$ is the identity operator on $\Z_p\langle x\rangle$. (See \cite[Lem. 2.1(3)]{CantatXie2018}.)

\subsection{Poonen's formula of the flow}

We first recap a theorem of Bell and Poonen \cite{Bell2006,Bell2006:corrigendum,Poonen2014}, specialized over $\Z_p$, to state the object of analysis.
\begin{theorem}[Bell--Poonen]
    Suppose $f\colon\Z_p^m\to\Z_p^m$ is an analytic map such that $f(x)\equiv x\pmod{p}$. Then there is an analytic map $\Phi_f\colon\Z_p\times\Z_p^m\to\Z_p^m$ such that
    \begin{enumerate}[(a)]
        \item $\Phi_f(n,x)=f^n(x)$ whenever $n\in\Z_{\geq 0}$ and $x\in\Z_p^m$; and
        \item $\Phi_f(s+t,x)=\Phi_f(s,\Phi_f(t,x))$ for all $s,t\in\Z_p$ and $x\in\Z_p^m$.
    \end{enumerate}
    Moreover, $f$ is invertible.
\end{theorem}

The map $\Phi_f$ is called the (Bell--Poonen) \emph{flow} generated by $f$.

\begin{proof}
    According to \cite{Poonen2014} (see the proof of its Theorem 1), the following Mahler interpolation
    \begin{equation}
        \label{eqn:flow-formula-by-Poonen}
        \Phi_f(t,x)=\sum_{j=0}^\infty t(t-1)\cdots(t-j+1)\cdot\frac1{j!}\left((T_f-\mathrm{Id})^j(\mathrm{id})\right)(x)
    \end{equation}
    defines an analytic map $\Z_p\times\Z_p^m\to\Z_p^m$. Indeed, each summand has the Gauss norm $\leq p^{j/(p-1)}\cdot p^{-j}\leq p^{-j/2}$, so the partial sums define analytic maps $\Z_p\times\Z_p^m\to\Z_p^m$ and they converge to a limit analytic map $\Phi_f$.     
    Moreover, we have $\Phi_f(n,x)=f^n(x)$ for $n=0,1,2,\cdots$, by the following computation (see \cite{Poonen2014}):
    \begin{align*}
        \Phi_f(n,x) &= \sum_{j=0}^n n(n-1)\cdots(n-j+1)\cdot\frac1{j!}\left((T_f-\mathrm{Id})^j(\mathrm{id})\right)(x) \\
        &= \left(\left(\sum_{j=0}^n\binom{n}{j}(T_f-\mathrm{Id})^j\right)(\mathrm{id})\right)(x) \\
        &= \left((\mathrm{Id}+(T_f-\mathrm{Id}))^n(\mathrm{id})\right)(x)=\left(T_f^n(\mathrm{id})\right)(x)=f^n(x).
    \end{align*}

    That $\Phi_f(s+t,x)=\Phi_f(s,\Phi_f(t,x))$ is clear if $s,t\in\Z_{\geq 0}$. By density, the result still holds for any $s,t\in\Z_p$. In particular, the map $g(x)=\Phi_f(-1,x)$ is the inverse of $f$.
\end{proof}

A finer analysis on the Gauss norm of each term in \eqref{eqn:flow-formula-by-Poonen} derives the followings. The 2nd term ($j=2$) has Gauss norm $\leq p^{-2}$ (as $\frac12\in\Z_p$); the 3rd term ($j=3$) has Gauss norm $\leq p\cdot p^{-3}=p^{-2}$; and the 4th term and beyond ($j\geq 4$) has the Gauss norm $\leq p^{-j/2}\leq p^{-2}$. Therefore, if we reduce the flow $\Phi_f$ modulo $p^2$ then we only look at the 0th and 1st terms:
\begin{align*}
    \Phi_f(t,x) &\equiv \binom{t}{0}\cdot\left((T_f-\mathrm{Id})^0(\mathrm{id})\right)(x) + \binom{t}{1}\cdot\left((T_f-\mathrm{Id})^1(\mathrm{id})\right)(x) \pmod{p^2} \\
    &= x + (f(x)-x)\cdot t.
\end{align*}
That is,
\begin{equation}
    \label{eqn:flow-modulo-p^2}
    \Phi_f(t,x) \equiv x + (f(x)-x)\cdot t\pmod{p^2}.
\end{equation}

\subsection{Inverse Function Theorem}

The following is immediate from \eqref{eqn:flow-modulo-p^2}.

\begin{proposition}
    \label{lem:invertibly-linear-modulo-p-is-invertible}
    Let $f\colon\Z_p^m\to\Z_p^m$ be an analytic map which is modulo $p$ affine-linear, i.e.,
    \begin{equation}
        \label{eqn:affine-linear-mod-p}
        f(x)\equiv A.x+\mathbf{b}\pmod{p},
    \end{equation}
    by an invertible matrix $A\in\mathsf{GL}_m(\Z_p)$ and a vector $\mathbf{b}\in\Z_p^m$. Then $f$ has an analytic inverse $f^{-1}$ such that
    \begin{equation}
        \label{eqn:mod-p-inversion}
        f^{-1}(x)\equiv A^{-1}.(x-\mathbf{b})\pmod{p}.
    \end{equation}
\end{proposition}
\begin{proof}
    Let $h(x)=A.x+\mathbf{b}$ be the affine map such that $f\equiv h\pmod{p}$. It follows that $h^{-1}f\equiv\mathrm{id}\pmod{p}$, so by the Bell--Poonen theorem we have a flow $\Phi_{h^{-1}f}$ generated by it. By the expansion \eqref{eqn:flow-modulo-p^2}, the flow is estimated modulo $p$ by
    \[\Phi_{h^{-1}f}(t,x)\equiv x\pmod{p}.\]
    In particular, the same estimate applies if we plug in $t=-1$. This yields
    \[f^{-1}(h(x))\equiv x\pmod{p}.\]
    Plug in $h^{-1}(x)$ into $x$ to get $f^{-1}(x)\equiv h^{-1}(x)\pmod{p}$, i.e., \eqref{eqn:mod-p-inversion}.
\end{proof}

\subsection{Local Minimality Results}

Consider the flows $\Phi_f$ and $\Phi_g$ by analytic maps $f,g\colon\Z_p^2\to\Z_p^2$. Fix a point $x_0$ and consider the map
\[(s,t)\in\Z_p^2\mapsto\Phi_f(s,\Phi_g(t,x_0))\in\Z_p^2,\]
which defines an analytic map into $x_0+(p\Z_p)^2$. We discuss when such a map yields a surjection onto $x_0+(p\Z_p)^2$, and derive minimality results from it.

\begin{proposition}
    \label{lem:local-minimality}
    Suppose $f,g\colon\Z_p^2\to\Z_p^2$ are analytic maps such that $f\equiv g\equiv\mathrm{id}\pmod{p}$. Suppose we have a point $x_0\in\Z_p^2$ such that the following determinant by column vectors is in $\Z_p^\times$:
    \begin{equation} 
        \label{eqn:indep-vectorial-part} 
        \det\left[\frac{f(x_0)-x_0}{p},\frac{g(x_0)-x_0}{p}\right]\in\Z_p^\times.
    \end{equation}
    Then the action of $\langle f,g\rangle$ on $x_0+(p\Z_p)^2$ is invariant and minimal.
\end{proposition}
\begin{proof}
    Because $f$ and $g$, as well as $f^{-1}$ and $g^{-1}$, are all modulo $p$ congruent to the identity map (for $f^{-1}$ and $g^{-1}$, reduce \eqref{eqn:flow-modulo-p^2} modulo $p$), we see that $\langle f,g\rangle$ must leave the disk $x_0+(p\Z_p)^2$ invariant.

    Consider the analytic map $F\colon\Z_p^2\to\Z_p^2$,
    \[F(s,t)=\frac1p\left(\Phi_f(s,\Phi_g(t,x_0))-x_0\right).\]
    By modulo $p^2$ expansion of the flows \eqref{eqn:flow-modulo-p^2}, this map is congruent to,
    \[\frac1p\left(\Phi_f(s,\Phi_g(t,x_0))-x_0\right)\equiv\frac{f(x_0)-x_0}{p}s+\frac{g(x_0)-x_0}{p}t\pmod{p},\]
    as functions of $s,t$. 
    By the assumption \eqref{eqn:indep-vectorial-part} on $x_0$ and Proposition \ref{lem:invertibly-linear-modulo-p-is-invertible}, we see that $F$ is invertible.

    To prove the minimality, let $x_1,x_2\in x_0+(p\Z_p)^2$ be any points. Find $(s_1,t_1)$ and $(s_2,t_2)$ in $\Z_p^2$ so that $(x_i-x_0)/p=F(s_i,t_i)$, i.e., $\Phi_f(s_i,\Phi_g(t_i,x_0))=x_i$ with $i=1,2$. These in particular yield that
    \[x_2 = \Phi_f(s_2,\Phi_g(t_2-t_1,\Phi_f(-s_1,x_1))).\]
    Now pick any integers $a$, $b$, and $c$ approximating $s_2$, $t_2-t_1$, and $-s_1$, respectively, in $\Z_p$. They yield a point $f^ag^bf^c(x_1)$ arbitrarily close to $x_2$. Hence the orbit of $x_1$ by $\langle f,g\rangle$ is dense in $x_0+(p\Z_p)^2$, proving the minimality.
\end{proof}

%The following is not used in the sequel yet appended for our reference.

\begin{proposition}
    \label{lem:conjugate-local-minimality}
    Let $f\colon\Z_p^2\to\Z_p^2$ be an analytic map such that $f\equiv\mathrm{id}\pmod{p}$. 
    Let $g\colon\Z_p^2\to\Z_p^2$ be an analytic map which is affine-linear modulo $p$, i.e., $g(x)\equiv A.x+\mathbf{b}\pmod{p}$ by an invertible matrix $A\in\mathsf{GL}_2(\Z_p)$ and a vector $\mathbf{b}\in\Z_p^2$. Suppose we have a point $x_0\in\Z_p^2$ such that the following determinant is in $\Z_p^\times$:
    \begin{equation}
        \label{eqn:twist-vector-condition}
        \det\left[\frac{f(x_0)-x_0}{p},A.\frac{f(g^{-1}(x_0))-g^{-1}(x_0)}{p}\right]\in\Z_p^\times.
    \end{equation}
    Then $\langle f,gfg^{-1}\rangle$ acts invariantly and minimally on $x_0+(p\Z_p)^2$.
\end{proposition}
\begin{proof}
    Let $v(x)=\frac1p(f(x)-x)$, which defines an analytic map $\Z_p^2\to\Z_p^2$ because $f\equiv\mathrm{id}\pmod{p}$. We show that
    \begin{equation}
        \label{eqn:mod-p^2-expansion-of-conjugation}
        gfg^{-1}(x)\equiv x + p\cdot A.v(g^{-1}(x))\pmod{p^2},
    \end{equation}
    so that we can apply Proposition \ref{lem:local-minimality} to $f$ and $gfg^{-1}$: \eqref{eqn:twist-vector-condition} verifies \eqref{eqn:indep-vectorial-part}.

    First, plug in $g^{-1}(x)$ into $x$ of $f(x)=x+pv(x)$:
    \[fg^{-1}(x) = g^{-1}(x) + pv(g^{-1}(x)).\]
    Let $Q(x)=\frac1p(g(x)-A.x-\mathbf{b})$ which also defines an analytic map $\Z_p^2\to\Z_p^2$. There, apply $g(x)=A.x+\mathbf{b}+pQ(x)$ on both sides:
    \begin{align*}
        gfg^{-1}(x) &= g(g^{-1}(x) + pv(g^{-1}(x))) \\
        &= A.\left(g^{-1}(x) + pv(g^{-1}(x))\right) + \mathbf{b} + pQ(g^{-1}(x)+pv(g^{-1}(x))) \\
        &\equiv A.g^{-1}(x) + pA.v(g^{-1}(x)) + \mathbf{b} + pQ(g^{-1}(x)) \pmod{p^2} \\
        &=A.g^{-1}(x) + \mathbf{b} + pQ(g^{-1}(x)) + pA.v(g^{-1}(x)) \\
        &= g(g^{-1}(x)) + pA.v(g^{-1}(x)) \\
        &= x + p A.v(g^{-1}(x)).
    \end{align*}
    This proves \eqref{eqn:mod-p^2-expansion-of-conjugation}.
\end{proof}

\section{Computation Gadgets}
\label{sec:computation-gadgets}

In this section, we introduce some computation gadgets required for the $p$-adic analysis of powers of maps $s_xs_y$, $s_ys_z$, or $s_zs_x$. All the results and proofs in this section are more or less elementary and perhaps can be skipped during the first reading. We put special emphasis on Propositions \ref{lem:trace-of-rr-modulo-p-uniqueness}, \ref{lem:companion-power-estimate}, and \ref{lem:companion-power-estimate-parab}, as they are referred frequently later.

\subsection{Monic Chebyshev Polynomials}

The Chebyshev polynomials of the 1st and 2nd kind respectively refer to the polynomial expressions of $\cos(N\theta)$ and $\sin((N+1)\theta)/\sin\theta$ by $\cos\theta$. They yield integral polynomials of degree $N$ (if $N\geq 0$) whose top coefficients are, respectively, $2^{N-1}$ and $2^N$.

An alternative version of these families of polynoimals, denoted by $(T_N)_{N\in\Z}$ and $(U_N)_{N\in\Z}$, is defined as follows.
\begin{definition}[Monic Chebyshev polynomials]
    \label{def:chebyshev}
    Define $T_N(x)$ and $U_N(x)$ the unique polynomials in $\Z[x]$ that satisfy
    \begin{align*}
        T_N\left(\lambda+\frac1\lambda\right) &= \lambda^N + \frac1{\lambda^N}, \\
        U_N\left(\lambda+\frac1\lambda\right) &= \frac{\lambda^{N+1}-\lambda^{-N-1}}{\lambda-\lambda^{-1}},
    \end{align*}
    for $N\in\Z$. Equivalently, this is a family of polynomials defined by recurrence relations
    \begin{align*}
        T_N(x) &= xT_{N-1}(x) - T_{N-2}(x), \\
        U_N(x) &= xU_{N-1}(x) - U_{N-2}(x),
    \end{align*}
    for all $N\in\Z$, with initials $(T_{-1}(x),T_0(x))=(x,2)$ and $(U_{-1}(x),U_0(x))=(0,1)$. 
    If $N\geq 0$, we say $T_N$ and $U_N$ respectively the \emph{monic Chebyshev polynomial of 1st} and \emph{2nd kind}, of degree $N$.
\end{definition}

According to A. F. Horadam \cite{Horadam2002}, polynomials $T_N$ are called \emph{Vieta--Lucas polynomials} and polynomials $U_N$ are called \emph{Vieta--Fibonacci polynomials}. We can recover the usual Chebyshev polynomials by $\frac12T_N(2x)$ and $U_N(2x)$: that is, we have $\frac12T_N(2\cos\theta)=\cos(N\theta)$ and $U_N(2\cos\theta)=\sin((N+1)\theta)/\sin\theta$.

By the definition, we have $T_{-N}(x)=T_N(x)$ and $U_{-N-2}(x) = - U_N(x)$ for all $N\in\Z$. For $N>0$, $T_N(x)$ and $U_N(x)$ are monic polynomials of degree $N$. Some first examples are: $T_0(x)=2$, $T_1(x)=T_{-1}(x)=z$, $U_0(x)=1$, $U_{-1}(x)=0$, and $U_{-2}(x)=-1$.

\begin{proposition}
	\label{lem:companion-power}
    Consider the companion matrix
    \begin{equation}
    	\label{eqn:companion-matrix-definition}
	    C(x)=\begin{pmatrix} x & -1 \\ 1 & 0 \end{pmatrix}.
    \end{equation}
    Its power is
    \begin{equation}
    	\label{eqn:companion-matrix-power}
    	C(x)^N=\begin{pmatrix} U_N(x) & -U_{N-1}(x) \\ U_{N-1}(x) & -U_{N-2}(x)\end{pmatrix}.
	\end{equation}
\end{proposition}
\begin{proof}
    The recurrence relation for $(U_N)$ can be rewritten as
    \[\begin{pmatrix}
        U_{N} \\ U_{N-1}
    \end{pmatrix}=\begin{pmatrix}
        x & -1 \\ 1 & 0
    \end{pmatrix}\begin{pmatrix}
        U_{N-1} \\ U_{N-2}
    \end{pmatrix}=C(x)\begin{pmatrix}
        U_{N-1} \\ U_{N-2}
    \end{pmatrix}.\]
    That is, multiplying $C(x)$ adds 1 to the indices. Because
    \[\begin{pmatrix} U_{0} & -U_{-1} \\ U_{-1} & -U_{-2} \end{pmatrix}=I_2,\]
    we see that
    \[C(x)^N=C(x)^N\begin{pmatrix} U_{0} & -U_{-1} \\ U_{-1} & -U_{-2}  \end{pmatrix}=\begin{pmatrix} U_N & -U_{N-1} \\ U_{N-1} & -U_{N-2}\end{pmatrix}. \qedhere\]
\end{proof}

\begin{remark}
	Consider the composition map $s_ys_z$, which sends $(x,y,z)$ to
	\[\begin{pmatrix} x \\ y \\ z \end{pmatrix} \overset{s_z}\mapsto \begin{pmatrix} x \\ y \\ xy-z \end{pmatrix} \overset{s_y}\mapsto \begin{pmatrix} x \\ x(xy-z)-y \\ xy-z \end{pmatrix} = \begin{pmatrix} x \\ \begin{pmatrix} x^2-1 & -x \\ x & -1 \end{pmatrix}\begin{pmatrix} y \\ z \end{pmatrix} \end{pmatrix}.\]
	Because $U_2(x)=x^2-1$, $U_1(x)=x$, and $U_0(x)=1$, it is immediate to see that
	\[\begin{pmatrix} x^2-1 & -x \\ x & -1 \end{pmatrix}=\begin{pmatrix} x & -1 \\ 1 & 0 \end{pmatrix}^2=C(x)^2.\]
	So an idea is to ``identify'' $(s_ys_z)$ with $C(x)^2$, especially when we want to study the behavior of its powers. This may explain why we introduce monic Chebyshev polynomials and companion matrices.
\end{remark}

Denote the $N$-th power of the companion matrix by $C_N(x):=C(x)^N$. We then study its properties of derivatives.
\begin{proposition}[First Derivatives]
	If $x\neq\pm 2$, we have
	\begin{equation}
		\label{eqn:companion-power-order-1}
		C_N'(x) = \frac{N}{x^2-4}\begin{pmatrix} T_{N+1}(x) & -T_{N}(x) \\ T_{N}(x) & -T_{N-1}(x) \end{pmatrix} + \frac{U_{N-1}(x)}{x^2-4}\begin{pmatrix} -2 & x \\ -x & 2 \end{pmatrix}.
	\end{equation}
	If $x=\pm 2$, we have
	\begin{align}
		C_N(\pm 2) &= (\pm 1)^N\begin{pmatrix} N+1 & \mp N \\ \pm N & 1-N \end{pmatrix}, \label{eqn:companion-power-order-0-border} \\
		C_N'(\pm 2) &= (\pm 1)^{N+1}\begin{pmatrix} \binom{N+2}{3} & \mp\binom{N+1}{3} \\ \pm\binom{N+1}{3} & -\binom{N}{3} \end{pmatrix}. \label{eqn:companion-power-order-1-border}
	\end{align}
\end{proposition}
\begin{proof}
	Plug in $x=2\cos\theta$. Then we have
	\begin{equation}
		\label{eqn:companion-power-order-0-trig}
		C_N(2\cos\theta) = \frac1{\sin\theta}\begin{pmatrix}\sin((N+1)\theta) & -\sin(N\theta) \\ \sin(N\theta) & -\sin((N-1)\theta) \end{pmatrix}.
	\end{equation}
	Take the derivative by $\theta$ and divide out $(2\cos\theta)'=-2\sin\theta$ from both sides. Then
	\begin{align*}
		C_N'(2\cos\theta)
		&= \frac{-N}{2\sin^2\theta}\begin{pmatrix} \cos((N+1)\theta) & -\cos(N\theta) \\ \cos(N\theta) & -\cos((N-1)\theta) \end{pmatrix} \\
		&\quad + \frac{\sin(N\theta)}{2\sin^3\theta}\begin{pmatrix} 1 & -\cos\theta \\ \cos\theta & -1 \end{pmatrix},
	\end{align*}
	by computation. As $\cos(N\theta)=\frac12T_N(2\cos\theta)$ and $\sin(N\theta)=\sin\theta U_{N-1}(2\cos\theta)$,
	\begin{align*}
		C_N'(2\cos\theta)&=\frac{-N}{2\sin^2\theta}\begin{pmatrix} \frac12T_{N+1} & -\frac12T_N \\ \frac12T_N & -\frac12T_{N-1} \end{pmatrix} + \frac{U_{N-1}}{2\sin^2\theta}\begin{pmatrix} 1 & -\cos\theta \\ \cos\theta & -1 \end{pmatrix} \\
		&= \frac{N}{-4\sin^2\theta}\begin{pmatrix} T_{N+1} & -T_N \\ T_N & -T_{N-1} \end{pmatrix} + \frac{U_{N-1}}{-4\sin^2\theta}\begin{pmatrix} -2 & 2\cos\theta \\ -2\cos\theta & 2 \end{pmatrix}.
	\end{align*}
	Substitute back $-4\sin^2\theta=x^2-4$ and $2\cos\theta=x$ to get \eqref{eqn:companion-power-order-1}.
	
	To get \eqref{eqn:companion-power-order-0-border}, we note a variant of \eqref{eqn:companion-power-order-0-trig}:
	\[C_N(\pm 2\cos\theta)=\frac{(\pm 1)^N}{\sin\theta}\begin{pmatrix}\sin((N+1)\theta) & \mp\sin(N\theta) \\ \pm\sin(N\theta) & -\sin((N-1)\theta) \end{pmatrix}.\]
	Send $\theta\to 0$ to get \eqref{eqn:companion-power-order-0-border}. To get \eqref{eqn:companion-power-order-1-border}, observe that
	\[C_N(\pm 2\cos\theta) = C_N(\pm 2) \mp C_N'(\pm 2)\theta^2 + O(\theta^4),\]
	so we have
	\[C_N'(\pm 2) = \pm\lim_{\theta\to 0}\frac{C_N(\pm 2)-C_N(\pm 2\cos\theta)}{\theta^2}.\]
	By computation, we get \eqref{eqn:companion-power-order-1-border} as a result. A key fact is that $(N\sin\theta-\sin(N\theta))/\theta^3\to\binom{N+1}{3}$ as $\theta\to 0$.
\end{proof}

\subsection{Traces of Rational Rotations}

Suppose $x$ is in a field $k$ and we can write $x=\lambda+\lambda^{-1}$ by some $\lambda\in\overline{k}$. Then $\lambda$ and $\lambda^{-1}$ are precisely the eigenvalues of the companion matrix $C(x)$ \eqref{eqn:companion-matrix-definition}. If $\lambda$ can be set close to an $N$-th root of unity, then $C(x)^N$ will be close to the identity.

To analytically elaborate this idea, we establish some terms to describe it.
\begin{definition}[Trace of Rational Rotation]
	Let $k$ be a field and let $\zeta$ be a root of unity in $\overline{k}$. If an element $x$ of $k$ can be written in the form $x=\zeta+\zeta^{-1}$ then $x$ is called a \emph{trace of rational rotation}. The \emph{order} of $x$ is the order of $\zeta$, i.e., the minimal integer $N\geq 1$ such that $\zeta^N=1$.
\end{definition}

The order of a trace of rational rotation is well-defined, by the following
\begin{lemma}
    \label{lem:uniqueness-of-trace-of-rr-expression}
	Suppose $\lambda,\mu$ are nonzero elements of a field $k$. If $\lambda+\lambda^{-1}=\mu+\mu^{-1}$, then either $\lambda=\mu$ or $\lambda=\mu^{-1}$.
\end{lemma}
\begin{proof}
	Note that $(\lambda+\lambda^{-1})-(\mu+\mu^{-1})=(\lambda-\mu)(1-(\lambda\mu)^{-1})$.
\end{proof}

So if $x=\lambda+\lambda^{-1}=\mu+\mu^{-1}$ and $\lambda$ has order $r$, then $\mu$ has the same order.

\begin{lemma}
	\label{lem:relation-trace-of-rr-Chebyshev-poly}
	\begin{enumerate}[(a)]
		\item If $x$ is a trace of rational rotation of order $r>2$, then $C(x)^r=I_2$.
		\item Suppose $N>1$. Then $T_N(x)=x$ iff $x$ is a trace of rational rotation of an order dividing $(N-1)$ or $(N+1)$.
	\end{enumerate}
\end{lemma}
\begin{proof}
	(a) Write $x=\zeta+\zeta^{-1}$ with $\zeta^r=1$. Then as $r>2$, $\zeta\neq\zeta^{-1}$, so $C(x)$ is diagonalizable and has eigenvalues $\zeta$ and $\zeta^{-1}$. Taking the $r$-th power of $C(x)$ then yields a diagonalizable matrix of eigenvalues 1 and 1, i.e., the identity matrix $I_2$.
	
	(b) Write $x=\lambda+\lambda^{-1}$. We have $T_N(x)=x$ iff $\lambda^N+\lambda^{-N}=\lambda+\lambda^{-1}$ iff $\lambda^N=\lambda^{\pm 1}$ iff $\lambda^{N+1}=1$ or $\lambda^{N-1}=1$. The last condition is equivalent to that $\lambda$ is a root of unity whose order divides $(N-1)$ or $(N+1)$.
\end{proof}

In part (a), note that the only traces of rational rotations of order $\leq 2$ are $2=1+1$ and $-2=(-1)+(-1)$. Indeed, $C(\pm 2)$ are unipotent.

\subsection{Contraction Properties of monic Chebyshev Polynomials}

Part (b) of Lemma \ref{lem:relation-trace-of-rr-Chebyshev-poly} can be used to establish the following
\begin{proposition}
	\label{lem:trace-of-rr-modulo-p-uniqueness}
	Every modulo $p$ residue class of $\Z_p$ contains a unique trace of rational rotation whose order divides $(p-1)$ or $(p+1)$.
\end{proposition}
\begin{proof}
By Lemma \ref{lem:relation-trace-of-rr-Chebyshev-poly}(b), it is equivalent to say that each residue class $x_0+p\Z_p$ contains a unique fixed point of the map $x\mapsto T_p(x)$ (where $T_p$ is the monic Chebyshev polynomial of degree $p$, see Definition \ref{def:chebyshev}). This leads us to study some contraction properties of this map.

\begin{lemma}
    \label{lem:mod-p-reduction-of-Tp}
    We have $T_p(x)\equiv x^p\pmod{p}$ as functions of $x\in\Z_p$.
\end{lemma}
\begin{proof}
	View $x=\lambda+\lambda^{-1}$. Apply the binomial theorem on $(\lambda+\lambda^{-1})^p$, so that
	\[(\lambda+\lambda^{-1})^p=\sum_{j=0}^p\binom{p}{j}\lambda^{p-2j}=\sum_{j=0}^{(p-1)/2}\binom{p}{j}(\lambda^{p-2j}+\lambda^{2j-p}).\]
	This collects to
	\begin{equation}
    	\label{eqn:Tj-to-xj-transit}
	    x^p=\sum_{j=0}^{(p-1)/2}\binom{p}{j}T_{p-2j}(x).
	\end{equation}
	Because $\binom{p}{j}$ is a multiple of $p$ if $0<j\leq\frac12(p-1)$, the conclusion follows.
\end{proof}

\begin{lemma}
	\label{lem:Tp-contraction}
	For two points $y_1$ and $y_2$ in the same modulo $p$ residue class $y+p\Z_p$, we have
	\[|T_p(y_1)-T_p(y_2)|_p\leq p^{-1}|y_1-y_2|_p.\]
	That is, $T_p$ is a $p^{-1}$-contraction on each disk $y+p\Z_p$.\end{lemma}
\begin{proof}
	Lemma \ref{lem:mod-p-reduction-of-Tp} shows that there is an integer-coefficient polynomial $Q(x)$ such that $T_p(x)=x^p+pQ(x)$. Thanks to the power estimates $|y_1^j-y_2^j|_p\leq|y_1-y_2|_p$, we have $|Q(y_1)-Q(y_2)|_p\leq|y_1-y_2|_p$. It follows that
	\begin{align*}
		|T_p(y_1)-T_p(y_2)|_p &= |(y_1^p+pQ(y_1))-(y_2^p+pQ(y_2))|_p \\
		&\leq\max\left(|y_1^p-y_2^p|_p,|p(Q(y_1)-Q(y_2))|_p\right) \\
		&\leq\max\left(|y_1^p-y_2^p|_p,p^{-1}|y_1-y_2|_p\right).
	\end{align*}
	We estimate $|y_1^p-y_2^p|_p$ as follows:
	\begin{align*}
		|y_1^p-y_2^p|_p = |(y_2+(y_1-y_2))^p-y_2^p|_p &= \left|\sum_{j=1}^p\binom{p}{j}(y_1-y_2)^jy_2^{p-j}\right|_p \\
		&\leq\max_{1\leq j\leq p}\left|\binom{p}{j}(y_1-y_2)^jy_2^{p-j}\right|_p \\
		&\leq\max\left(|y_1-y_2|_p^p,\max_{1\leq j<p}p^{-1}|y_1-y_2|_p^j\right) \\
		&=p^{-1}|y_1-y_2|_p,
	\end{align*}
	where the last line holds if $|y_1-y_2|_p\leq p^{-1}$. So the estimate follows.
\end{proof}

By Lemma \ref{lem:mod-p-reduction-of-Tp}, as $x_0^p\equiv x_0\pmod{p}$ for all $x_0\in\Z_p$, we have $T_p(x)\in x_0+p\Z_p$ for any $x\in x_0+p\Z_p$. Lemma \ref{lem:Tp-contraction} then shows that $x\mapsto T_p(x)$ is a contraction map on $x_0+p\Z_p$, and therefore contains a unique fixed point $x_1\in x_0+p\Z_p$ which is a trace of rational rotation whose order divides $(p-1)$ or $(p+1)$.
\end{proof}

\subsection{Estimates of Powers of Companion matrices}

If $p$ is an odd prime, the least common multiple of $(p-1)$ and $(p+1)$ is $(p^2-1)/2$. So if $x_1$ is a trace of rational rotation whose order divides $(p-1)$ or $(p+1)$, we have $C(x_1)^{(p^2-1)/2}=I_2$. By this, we show that certain powers of the companion matrix $C(x_0)$ can be guaranteed to be very close to the identity.
\begin{proposition}
	\label{lem:companion-power-estimate}
	Let $x_0\in\Z_p$, $x_0\not\equiv\pm2\pmod{p}$. Let $x_1$ be the unique fixed point of $T_p$ in $x_0+p\Z_p$. Let $N=(p^2-1)/2$. Then as functions of $u\in\Z_p$, we have
    \begin{align}
        C(x_0+pu)^N &\equiv I_2 + \frac{N}{x_0^2-4}\begin{pmatrix} x_0 & -2 \\ 2 & -x_0 \end{pmatrix}(x_0-x_1+pu)\pmod{p^2}, \label{eqn:companion-power-estimate-1}\\
        C(x_0+pu)^{pN}&\equiv I_2 + \frac{pN}{x_0^2-4}\begin{pmatrix} x_0 & -2 \\ 2 & -x_0 \end{pmatrix}(x_0-x_1+pu)\pmod{p^3}. \label{eqn:companion-power-estimate-2}
    \end{align}
\end{proposition}
\begin{remark}
    The number $N=(p^2-1)/2$ can be replaced to any multiple of the order of $x_1$ (as a trace of rational rotation) which is relatively prime to $p$.
\end{remark}
\begin{proof}
	Let $C_N(x)=C(x)^N$. Write $x_1=\zeta+\zeta^{-1}$ where the order of $\zeta$ divides $N$ and greater than 2 (note that $x_1\neq\pm 2$). Then by Lemma \ref{lem:relation-trace-of-rr-Chebyshev-poly}(a), we have $C_N(x_1)=I_2$.
	
	Now we compute $C_N'(x_1)$. First, from the $(2,1)$-entry of $C_N(x_1)=I_2$, we have $U_{N-1}(x_1)=0$. Next, as $\zeta^N=1$, we have $T_{N+1}(x_1)=\zeta^{N+1}+\zeta^{-N-1}=\zeta+\zeta^{-1}=x_1$, as well as $T_{N-1}(x_1)=\zeta^{N-1}+\zeta^{1-N}=\zeta^{-1}+\zeta=x_1$; moreover, $T_N(x_1)=\zeta^N+\zeta^{-N}=2$. These suffice to invoke \eqref{eqn:companion-power-order-1} and have
	\[C_N'(x_1)=\frac{N}{x_1^2-4}\begin{pmatrix} x_1 & -2 \\ 2 & -x_1 \end{pmatrix}.\]
	
	By the Taylor expansion, as $C_N(x)$ has degree $N$,
	\[C_N(x_0+pu) = \sum_{r=0}^N\frac{C_N^{(r)}(x_1)}{r!}(x_0-x_1+pu)^r.\]
	But we know that $x_0-x_1+pu\equiv 0\pmod{p}$ as functions of $u$. Hence we reduce to
	\[C_N(x_0+pu) \equiv C_N(x_1) + C_N'(x_1)(x_0-x_1+pu)\pmod{p^2}.\]
	Plug in the computed results and get \eqref{eqn:companion-power-estimate-1}:
    \begin{align*}
        C_N(x_0+pu) &\equiv I_2 + \frac{N}{x_1^2-4}\begin{pmatrix} x_1 & -2 \\ 2 & -x_1 \end{pmatrix}(x_0-x_1+pu) \pmod{p^2} \\
        &\equiv I_2 + \frac{N}{x_0^2-4}\begin{pmatrix} x_0 & -2 \\ 2 & -x_0 \end{pmatrix}(x_0-x_1+pu)\pmod{p^2}.
    \end{align*}

    Next we show the estimate for the $pN$-th power \eqref{eqn:companion-power-estimate-2}. Set
    \[Y(u) = \frac{N}{x_0^2-4}\begin{pmatrix} x_0 & -2 \\ 2 & -x_0 \end{pmatrix}\cdot\frac{x_0-x_1+pu}{p},\]
    which is a matrix-valued analytic map on $u$. Then by \eqref{eqn:companion-power-estimate-1}, there exists a matrix-valued analytic map $E(u)$ such that
    \[C(x_0+pu)^{(p^2-1)/2} = I_2 + (pY(u)+p^2E(u)).\]
    Raising to power $p$ on both sides, we get
    \begin{align*}
        C(x_0+pu)^{p(p^2-1)/2} &\equiv I_2 + \binom{p}{1}(pY+p^2E) + \binom{p}{2}(pY+p^2E)^2\pmod{p^3} \\
        &\equiv I_2 + p^2Y\pmod{p^3}.
    \end{align*}
    This shows \eqref{eqn:companion-power-estimate-2}.
\end{proof}

To include $x_0\equiv\pm 2\pmod{p}$, we need a different power:
\begin{proposition}
	\label{lem:companion-power-estimate-parab}
	Let $p>3$. Suppose $x_0\in\Z_p$, $x_0\equiv\pm 2\pmod{p}$. Let $x_1=\pm 2$, with the sign chosen so that $x_0\equiv x_1\pmod{p}$. Then we have
    \begin{align}
        C(x_0+pu)^{2p}&\equiv I_2 + p\begin{pmatrix} 2 & -x_0 \\ x_0 & -2 \end{pmatrix}\pmod{p^2}, \label{eqn:companion-power-estimate-parab-1}\\
        C(x_0+pu)^{2p^2}&\equiv I_2 + p^2\begin{pmatrix} 2 & -x_0 \\ x_0 & -2 \end{pmatrix}\pmod{p^3} \label{eqn:companion-power-estimate-parab-2}
    \end{align}
	as functions of $u\in\Z_p$.
\end{proposition}
\begin{proof}
	Let $N=2p$. Then by \eqref{eqn:companion-power-order-1-border}, $C_N'(x_1)$ is a multiple of $p$. %$U_N''(x_1)$, $U_{N-1}''(x_1)$, and $U_{N-2}''(x_1)$ are all multiples of $p$. 
	Because of that, the Taylor expansion yields
	\[C_N(x_0+pu) \equiv C_N(x_1)\pmod{p^2}.\]
	It remains to compute $C_N(x_1)$ according to \eqref{eqn:companion-power-order-0-border} with $N=2p$:
    \begin{align*}
        C_N(x_1) = C_N(\pm 2) &= (\pm 1)^{2p}\begin{pmatrix} 2p+1 & \mp 2p \\ \pm 2p & 1-2p \end{pmatrix} \\
        &= I_2 + p\begin{pmatrix} 2 & -(\pm 2) \\ (\pm 2) & -2 \end{pmatrix} \\
        &\equiv I_2 + p\begin{pmatrix} 2 & -x_0 \\ x_0 & -2 \end{pmatrix}\pmod{p^2}.
    \end{align*}
    %because $px_1\equiv px_0\pmod{p^2}$. 
    For the $2p^2$-th power, we argue as in Proposition \ref{lem:companion-power-estimate}.
\end{proof}

\subsection{Distance} We finish the section by declaring a standard distance assigned to $\Z_p^3$, hence $X_D^\ast(\Z_p)$ too. 
Define the \emph{($\ell^\infty$-)distance} on $\Z_p^3$ as follows:
\[\dist((x,y,z),(x',y',z'))=\max(|x-x'|_p,|y-y'|_p,|z-z'|_p).\]
\begin{lemma}
    \begin{enumerate}[(a)]
        \item The closed $p^{-k}$-neighborhood of $(x,y,z)$ by $\dist$ is $(x+p^k\Z_p)\times(y+p^k\Z_p)\times(z+p^k\Z_p)$.
        \item The distance $\dist$ is discretely valued: $\dist(\mathbf{p},\mathbf{q})\in p^\Z$ whenever $\mathbf{p},\mathbf{q}\in\Z_p^3$.
        \item The automorphism group $\Aut(X_D^\ast)$ isometrically on $X_D^\ast(\Z_p)$.
    \end{enumerate}
\end{lemma}
\begin{proof}
    We only prove (c). It suffices to show that each Vieta involution is isometric, say $s_x$ is. Observe that
    \begin{align*}
        |(yz-x)-(y'z'-x')|_p&\leq\max(|x-x'|_p,|yz-y'z'|_p)\\&\leq\max(|x-x'|_p,|y-y'|_p,|z-z'|_p)
    \end{align*}
    whenever $(x,y,z)\in\Z_p^3$. Hence $\dist(s_x(\mathbf{p}),s_x(\mathbf{q}))\leq\dist(\mathbf{p},\mathbf{q})$ for any $\mathbf{p},\mathbf{q}\in\Z_p^3$. Apply the inequality to $\mathbf{p}\leftarrow s_x(\mathbf{p})$ and $\mathbf{q}\leftarrow s_x(\mathbf{q})$. This shows that the inequality is really an equality.
\end{proof}

The topology on $X_D^\ast(\Z_p)\subset\Z_p^3$ by the distance $\dist$ is the unique topology that is inherited from the product topology on $\Z_p^3$. This may justify why we specifically choose this $\dist$ on $X_D^\ast(\Z_p)$. 

\section{Local Analysis of Stabilizers}
\label{sec:local-analysis-stabilizers}

Let
\begin{equation}
    \label{eqn:reduction-map}
    \red_\ell\colon X_D^\ast(\Z_p)\to X_D^\ast(\Z_p/p^\ell\Z_p)
\end{equation}
be the reduction map modulo $p^\ell$. We call its fibers \emph{level $\ell$ polydisks}. In this section, given a level 1 polydisk $U$, we study the action of the stabilizer $\Stab_\Gamma(U)$ in $\Gamma$. In particular, we suggest some elements of the stabilizer group and study the local action of these elements.

\subsection{Examples of Stabilizers}

Consider the elements $s_ys_z$, $s_zs_x$, and $s_xs_y$ in $\Gamma$. They are elements that fix the $x$, $y$, and $z$ coordinate respectively. In effect, we can describe the action of each by the remaining two coordinates and the companion matrix $C({}\cdot{})$ \eqref{eqn:companion-matrix-definition}:
\begin{align*}
    s_ys_z\colon\begin{pmatrix} y \\ z \end{pmatrix} &\mapsto C(x)^2\begin{pmatrix} y \\ z \end{pmatrix}, \\
    s_zs_x\colon\begin{pmatrix} z \\ x \end{pmatrix} &\mapsto C(y)^2\begin{pmatrix} z \\ x \end{pmatrix}, \\
    s_xs_y\colon\begin{pmatrix} x \\ y \end{pmatrix} &\mapsto C(z)^2\begin{pmatrix} x \\ y \end{pmatrix}.
\end{align*}
(See the remark after Proposition \ref{lem:companion-power}.) Moreover, Propositions \ref{lem:companion-power-estimate} and \ref{lem:companion-power-estimate-parab} show that powers of companion matrices are almost the identity. So a sufficient power of these elements are suitable to preserve a level 1 polydisk.

We formulate this observation in the following
\begin{lemma}
    \label{lem:examples-of-stabilizers}
    Let $(x_0,y_0,z_0)\in X_D^\ast(\Z_p)$ be a point and let $U$ be a level 1 polydisk containing that point.
    \begin{enumerate}[(a)]
        \item If $x_0\not\equiv\pm 2\pmod{p}$, then $(s_ys_z)^{(p^2-1)/4}$ stabilizes $U$.
        \item If $x_0\equiv\pm 2\pmod{p}$, then $(s_ys_z)^p$ stabilizes $U$.
    \end{enumerate}
\end{lemma}
\begin{proof}
    Define $f_0=s_ys_z$. Note that this map multiplies $C(x)^2$ to the $(y,z)$ coordinates:
    \[f_0\colon\begin{pmatrix} y \\ z \end{pmatrix}\mapsto C(x)^2\begin{pmatrix} y \\ z \end{pmatrix}.\]
    
    (a) If $x_0\not\equiv\pm 2\pmod{p}$, then by Proposition \ref{lem:companion-power-estimate}, we see that $C(x)^{(p^2-1)/2}$ is congruent to $I_2$ modulo $p$, whenever $x\equiv x_0\pmod{p}$. So for any $(x,y,z)\in U$, i.e., $(x,y,z)\equiv(x_0,y_0,z_0)\pmod{p}$,
    \begin{align*}
        f_0^{(p^2-1)/4}\begin{pmatrix} x \\ y \\ z \end{pmatrix} &= \begin{pmatrix} x \\ C(x)^{(p^2-1)/2}\begin{pmatrix} y \\ z \end{pmatrix}\end{pmatrix} 
        \equiv \begin{pmatrix} x_0 \\ I_2\begin{pmatrix} y_0 \\ z_0 \end{pmatrix}\end{pmatrix}=\begin{pmatrix} x_0 \\ y_0 \\ z_0 \end{pmatrix} \pmod{p},
    \end{align*}
    and $f_0^{(p^2-1)/4}(x,y,z)\in U$ follows.

    (b) If $x_0\equiv\pm 2\pmod{p}$, then by Proposition \ref{lem:companion-power-estimate-parab}, we see that $C(x)^{2p}$ is congruent to $I_2$ modulo $p$, whenever $x\equiv x_0\pmod{p}$. So for any $(x,y,z)\in U$,
    \begin{align*}
        f_0^{p}\begin{pmatrix} x \\ y \\ z \end{pmatrix} &= \begin{pmatrix} x \\ C(x)^{2p}\begin{pmatrix} y \\ z \end{pmatrix}\end{pmatrix} 
        \equiv \begin{pmatrix} x_0 \\ I_2\begin{pmatrix} y_0 \\ z_0 \end{pmatrix}\end{pmatrix}=\begin{pmatrix} x_0 \\ y_0 \\ z_0 \end{pmatrix} \pmod{p},
    \end{align*}
    and $f_0^{p}(x,y,z)\in U$ follows.
\end{proof}

Because the surface is symmetric by its coordinates, similar conclusion holds for powers of $s_zs_x$ with conditions of $y_0$, or powers of $s_xs_y$ with conditions on $z_0$.

\subsection{Parametrizing Polydisks}
\label{sec:parametrized-polydisks}

By the polynoimal
\begin{equation}
    \label{eqn:mother-polynomial}
    P(x,y,z) = x^2+y^2+z^2-xyz,
\end{equation}
we can describe the surface $X_D^\ast$ as the variety $(P(x,y,z)=D)$, removing singular points on it. Having said so, given any point $(x_0,y_0,z_0)\in X_D^\ast(\Z_p)$, one of its partial derivatives, $\partial_xP$, $\partial_yP$, or $\partial_zP$, evaluates to an invertible element of $\Z_p$ (i.e., those not congruent to 0 modulo $p$).

Suppose $(x_0,y_0,z_0)\in X_D^\ast(\Z_p)$ is a point such that $\partial_xP(x_0,y_0,z_0)\not\equiv 0\pmod{p}$, say. As $X_D^\ast$ is determined by the equation $P(x,y,z)=D$, the implicit function theorem then applies to yield an analytic map
\[\xi\colon(y_0+p\Z_p)\times(z_0+p\Z_p)\to(x_0+p\Z_p)\]
such that $\xi(y_0,z_0)=x_0$ and $P(\xi(y,z),y,z)=D$ for all $(y,z)\in(y_0+p\Z_p)\times(z_0+p\Z_p)$. Furthermore, we can do Taylor expansion of the analytic map $\xi(y_0+pu,z_0+pv)$ as a function of $(u,v)$, to get a modulo $p^2$ expansion of the map. We collect these as a lemma for later references.

\begin{lemma}
    \label{lem:parametrization-expansion}
    Let $U$ be a level 1 polydisk containing a point $(x_0,y_0,z_0)\in X_D^\ast(\Z_p)$. If $\partial_xP(x_0,y_0,z_0)=2x_0-y_0z_0\not\equiv 0\pmod{p}$, then there exists an analytic map $\xi\colon(y_0+p\Z_p)\times(z_0+p\Z_p)\to(x_0+p\Z_p)$ such that
    \begin{enumerate}[(i)]
        \item we have $P(\xi(y,z),y,z)=D$ whenever $(y,z)\in(y_0+p\Z_p)\times(z_0+p\Z_p)$, $\xi(y_0,z_0)=x_0$, and $\xi$ is the unique choice of such properties;
        \item the map
        \begin{equation}
            \label{eqn:polydisk-parametrization}
            \Psi\colon(u,v)\in\Z_p^2\mapsto(\xi(y_0+pu,z_0+pv),y_0+pu,z_0+pv)\in U
        \end{equation}
        is an analytic isomorphism with the inverse $\Psi^{-1}(x,y,z)=(\frac1p(y-y_0),\frac1p(z-z_0))$; and
        \item we have an expansion
        \begin{equation}
            \label{eqn:parametrization-expansion}
            \xi(y_0+pu,z_0+pv)\equiv x_0 -\frac{\partial_yP}{\partial_xP}pu-\frac{\partial_zP}{\partial_xP}pv \pmod{p^2}
        \end{equation}
        as functions of $(u,v)$. Here, the partial derivatives of $P$ are evaluated at the point $(x_0,y_0,z_0)$.
    \end{enumerate}
\end{lemma}

The parametrization \eqref{eqn:polydisk-parametrization} will be named as the \emph{parametrization of the polydisk $U$ based on the $(y,z)$-coordinates}, and the point $(y_0,z_0)$ or $(x_0,y_0,z_0)$ will be called the \emph{center} of this parametrization.

\subsection{Expansions of Stabilizers}

Via the parametrization $\Psi$ \eqref{eqn:polydisk-parametrization}, we now conjugate the known stabilizers as in Lemma \ref{lem:examples-of-stabilizers} by $\Psi$ and study their local actions.

The following assumption will be repeated, so we label them as follows (named after `Standing Assumptions'):
\begin{itemize}
    \item[\SA] Let $(x_0,y_0,z_0)\in X_D^\ast(\Z_p)$ and let $U$ be the level 1 polydisk containing $(x_0,y_0,z_0)$. Suppose $\partial_xP(x_0,y_0,z_0)\not\equiv 0\pmod{p}$ so that we have a pa\-ra\-me\-tri\-za\-tion $\Psi$ based on the $(y,z)$-coordinates, centered at $(y_0,z_0)$.
\end{itemize}

We will frequently refer to the partial derivatives of the polynomial $P$ \eqref{eqn:mother-polynomial}. Under the assumption \SA{}, all such partial derivatives are assumed to be evaluated at the point $(x_0,y_0,z_0)$.

\begin{lemma}[Powers of $s_ys_z$, when $x_0\equiv\pm 2\pmod{p}$]
    \label{lem:expansion-para-f}
    Assume \SA{}. Suppose $x_0\equiv\pm 2\pmod{p}$. Then by $f=\Psi^{-1}(s_ys_z)^p\Psi$ we have
        \begin{align}
            f\begin{pmatrix} u \\ v \end{pmatrix} &\equiv \begin{pmatrix} u \\ v \end{pmatrix}+\begin{pmatrix} \partial_yP \\ -\partial_zP \end{pmatrix}\pmod{p}. \label{eqn:expansion-para-f-mod-p}
        \end{align}
\end{lemma}
\begin{proof}
    Throughout the proof, $\xi=\xi(y_0+pu,z_0+pv)$, an analytic map of $(u,v)$. Observe that $\xi\equiv x_0\pmod{p}$ by \eqref{eqn:parametrization-expansion}.

    As $(s_ys_z)^p$ sends $\begin{pmatrix} y \\ z \end{pmatrix}$ to $C(x)^{2p}\begin{pmatrix} y \\ z \end{pmatrix}$, we have
    \begin{align*}
        f\begin{pmatrix} u \\ v \end{pmatrix} &= \frac1p\left[C(\xi)^N\begin{pmatrix} y_0+pu \\ z_0+pv \end{pmatrix}-\begin{pmatrix} y_0 \\ z_0 \end{pmatrix}\right] \\
        &= C(\xi)^N\begin{pmatrix} u \\ v \end{pmatrix} + \frac{C(\xi)^N-I_2}{p}\begin{pmatrix} y_0 \\ z_0 \end{pmatrix}.
    \end{align*}
    By Proposition \ref{lem:companion-power-estimate-parab}, we have
    \begin{align*}
        f\begin{pmatrix} u \\ v \end{pmatrix} &= C(\xi)^N\begin{pmatrix} u \\ v \end{pmatrix} + \frac{C(\xi)^N-I_2}{p}\begin{pmatrix} y_0 \\ z_0 \end{pmatrix} \\
        &\equiv \begin{pmatrix} u \\ v \end{pmatrix} + \begin{pmatrix} 2 & -x_0 \\ x_0 & -2 \end{pmatrix}\begin{pmatrix} y_0 \\ z_0 \end{pmatrix}\pmod{p} \\
        &= \begin{pmatrix} u \\ v \end{pmatrix} + \begin{pmatrix} 2y_0-x_0z_0 \\ -2z_0+x_0y_0 \end{pmatrix} = \begin{pmatrix} u \\ v \end{pmatrix}+\begin{pmatrix} \partial_yP \\ -\partial_zP \end{pmatrix}. \qedhere
    \end{align*}
\end{proof}

\begin{lemma}[Powers of $s_zs_x$ and $s_xs_y$]
    \label{lem:expansion-g-h}
    Assume \SA{}. Let $N=(p^2-1)/2$.
    \begin{enumerate}[(a)]
        \item Suppose $y_0\not\equiv\pm 2\pmod{p}$. Let $y_1$ be the unique fixed point of $T_p$ in $y_0+p\Z_p$. Let $c_1$ be the number
        \[c_1=-N\frac{\partial_xP}{y_0^2-4}.\]
        Then by $g=\Psi^{-1}(s_zs_x)^{N/2}\Psi$ we have
        \begin{align}
            \label{eqn:expansion-g-mod-p}
            g\begin{pmatrix} u \\ v \end{pmatrix}&\equiv\begin{pmatrix} 1 & 0 \\ c_1 & 1 \end{pmatrix}\begin{pmatrix} u \\ v \end{pmatrix} + \frac{y_0-y_1}{p}\begin{pmatrix} 0 \\ c_1 \end{pmatrix}\pmod{p}, \\
            \label{eqn:expansion-gp-mod-pp}
            g^p\begin{pmatrix} u \\ v \end{pmatrix}&\equiv\begin{pmatrix} 1 & 0 \\ pc_1 & 1 \end{pmatrix}\begin{pmatrix} u \\ v \end{pmatrix} + \frac{z_0-z_1}{p}\begin{pmatrix} 0 \\ pc_1 \end{pmatrix}\pmod{p^2}.
        \end{align}
        \item Suppose $z_0\not\equiv\pm 2\pmod{p}$. Let $z_1$ be the unique fixed point of $T_p$ in $z_0+p\Z_p$. Let $c_2$ be the number
        \[c_2=N\frac{\partial_xP}{z_0^2-4}.\]
        Then by $h=\Psi^{-1}(s_xs_y)^{N/2}\Psi$ we have
        \begin{align}
            \label{eqn:expansion-h-mod-p}
            h\begin{pmatrix} u \\ v \end{pmatrix}&\equiv\begin{pmatrix} 1 & c_2 \\ 0 & 1 \end{pmatrix}\begin{pmatrix} u \\ v \end{pmatrix} + \frac{z_0-z_1}{p}\begin{pmatrix} c_2 \\ 0 \end{pmatrix}\pmod{p}, \\
            \label{eqn:expansion-hp-mod-pp}
            h^p\begin{pmatrix} u \\ v \end{pmatrix}&\equiv\begin{pmatrix} 1 & pc_2 \\ 0 & 1 \end{pmatrix}\begin{pmatrix} u \\ v \end{pmatrix} + \frac{z_0-z_1}{p}\begin{pmatrix} pc_2 \\ 0 \end{pmatrix}\pmod{p^2}.
        \end{align}
    \end{enumerate}
\end{lemma}
\begin{proof}
    Denote by $\xi=\xi(y_0+pu,z_0+pv)$ throughout the proof. Note that $\xi\equiv x_0\pmod{p}$ as functions of $(u,v)$.	
	
	(a) We know that $(s_zs_x)^{N/2}=(s_zs_x)^{(p^2-1)/4}$ sends $(z,x)$ to
	\[(s_zs_x)^{N/2}\colon\begin{pmatrix} z \\ x \end{pmatrix}\mapsto C(y)^N\begin{pmatrix} z \\ x \end{pmatrix}=\begin{pmatrix} U_N(y)z-U_{N-1}(y)x \\ U_{N-1}(y)z-U_{N-2}(y)x \end{pmatrix}.\]
	If we plug in $(z,x)\leftarrow(z_0+pv,\xi)$ and apply $\Psi^{-1}$ to the result, we see that
	\[g\begin{pmatrix} u \\ v \end{pmatrix} = \begin{pmatrix} u \\ \frac1p(U_N(y_0+pu)(z_0+pv)-U_{N-1}(y_0+pu)\xi-z_0) \end{pmatrix}.\]
	By the expansion in Proposition \ref{lem:companion-power-estimate}, we have, with $c_0=N/(y_0^2-4)$,
	\begin{align*}
		U_N(y_0+pu) &\equiv 1 + c_0y_0(y_0-y_1+pu)\pmod{p^2}, \\
		U_{N-1}(y_0+pu) &\equiv 2c_0(y_0-y_1+pu)\pmod{p^2}.
	\end{align*}
	Together with $(y_0-y_1+pu)\equiv 0\pmod{p}$, we proceed as
	\begin{align*}
		&\frac{U_N(y_0+pu)(z_0+pv)-U_{N-1}(y_0+pu)\xi-z_0}{p} \\
		&\equiv\frac{(1+c_0y_0(y_0-y_1+pu))(z_0+pv)-z_0-2c_0(y_0-y_1+pu)\xi}{p} \pmod{p} \\
		&\equiv c_0y_0\left(\frac{y_0-y_1}{p}+u\right)z_0+v-2c_0\left(\frac{y_0-y_1}{p}+u\right)x_0 \pmod{p} \\
		&= v - (2x_0-y_0z_0)c_0\left(\frac{y_0-y_1}{p}+u\right) \\
		&= v + c_1\cdot\left(\frac{y_0-y_1}{p}+u\right).
	\end{align*}
	Hence \eqref{eqn:expansion-g-mod-p} follows. Similarly, to prove \eqref{eqn:expansion-gp-mod-pp} we pay attention to
    \[g^p\begin{pmatrix} u \\ v \end{pmatrix}=\begin{pmatrix} u \\ \frac1p(U_{Np}(y_0+pu)(z_0+pv)-U_{Np-1}(y_0+pu)\xi-z_0)\end{pmatrix},\]
    and
    \begin{align*}
        U_{Np}(y_0+pu) &\equiv 1+pc_0y_0(y_0-y_1+pu)\pmod{p^3}, \\
        U_{Np-1}(y_0+pu) &\equiv 2pc_0(y_0-y_1+pu)\pmod{p^3},
    \end{align*}
    to perform a similar computation as if $c_1$ is replaced to $pc_1$.

    (b) We do the same computation, but we pay attention to
    \[h\begin{pmatrix} u \\ v \end{pmatrix} = \begin{pmatrix} \frac1p(U_{N-1}(z_0+pv)\xi-U_{N-2}(z_0+pv)(y_0+pu)-y_0) \\ v \end{pmatrix}\]
    and, with $c_0=N/(z_0^2-4)$,
    \begin{align*}
        U_{N-1}(z_0+pv) &\equiv 2c_0(z_0-z_1+pv) \pmod{p^2}, \\
        U_{N-2}(z_0+pv) &\equiv -1+c_0z_0(z_0-z_1+pv) \pmod{p^2}.
    \end{align*}
    Raising to power $p$ is also similar.
\end{proof}

\section{A Strict Move in a Level 1 Polydisk}
\label{sec:counting-argument}

In this section, we establish that there exists a point $\mathbf{p}\in X_D^\ast(\Z_p)$ and an element $\gamma\in\Gamma$ such that $\dist(\mathbf{p},\gamma.\mathbf{p})=p^{-1}$. The proof differs according to $\left(\frac{D-4}{p}\right)=1$ or $D\equiv 0\pmod{p^2}$; for the former, we have an explicit choice of such a pair $(\mathbf{p},\gamma)$.

In case of $D\equiv 0\pmod{p^2}$ yet $\left(\frac{D-4}{p}\right)\neq 1$, we have $p\equiv 3\pmod{4}$, and we can use D. Martin's proof \cite{Martin2025} of modulo $p$ divisibility of $\Gamma$-orbits, which easily generalizes to modulo $p^2$ divisibiity of orbits. Together with a counting argument, we show that \emph{any} point $\mathbf{p}\in X_D^\ast(\Z_p)$ has a $\gamma\in\Gamma$ of the desired property.

\subsection{If $(D-4)$ is a nonzero quadratic residue} 
In this case, under certain conditions on the point $\mathbf{p}\in X_D^\ast(\Z_p)$, one can manually find $\gamma\in\Gamma$ such that $\dist(\mathbf{p},\gamma.\mathbf{p})=p^{-1}$.

\begin{proposition}
    \label{lem:pigeonhole-mod-pp-1}
    If $\mathbf{p}\in X_D^\ast(\Z_p)$ is a point such that
    \begin{itemize}
        \item one of its coordinate is congruent to $\pm 2$ modulo $p$ and
        \item all partial derivatives of $P(x,y,z)=x^2+y^2+z^2-xyz$ at $\mathbf{p}$ are nonzero modulo $p$,
    \end{itemize}
    then there exists $\gamma\in\Gamma$ such that $\dist(\mathbf{p},\gamma.\mathbf{p})=p^{-1}$.
\end{proposition}
\begin{proof}
    Let $\mathbf{p}=(x_0,y_0,z_0)$ be such that $x_0\equiv\pm 2\pmod{p}$, say. Then we can specifically say that $\gamma=(s_ys_z)^p$ works. Indeed, by Lemma \ref{lem:expansion-para-f} and our assumptions on partial derivatives, we have $\Psi^{-1}(\gamma.\mathbf{p})\not\equiv\Psi^{-1}\mathbf{p}\pmod{p}$. This implies that $\gamma.\mathbf{p}\not\equiv\mathbf{p}\pmod{p^2}$. But Lemma \ref{lem:examples-of-stabilizers} implies $\gamma.\mathbf{p}\equiv\mathbf{p}\pmod{p}$.
\end{proof}

See Lemma \ref{lem:special-point} and Proposition \ref{lem:minimal-polydisk-exceptional} to see that there is point $\mathbf{p}\in X_D^\ast(\Z_p)$ satisfying the above hypotheses.

\iffalse
\begin{remark}
    \begin{enumerate}
        \item If $(D-4)$ is a nonzero quadratic residue modulo $p$, then at least one of the points $(2,\pm\sqrt{D-4},0)\in X_D^\ast(\Z_p)$ meets the hypotheses. The choice of sign depends on the choice of $\sqrt{D-4}\in\Z_p$, so we may assume that $(2,\sqrt{D-4},0)$ is meeting the hypotheses.
        \item If $\Gamma$ acts transitively on $X_D^\ast$ modulo $p$, then we can upgrade the statement as, for every point $\mathbf{p}\in X_D^\ast(\Z_p)$ there exists $\gamma\in\Gamma$ such that $\dist(\mathbf{p},\gamma.\mathbf{p})=p^{-1}$. Just take $\alpha\in\Gamma$ such that $\alpha.\mathbf{p}\equiv(2,\sqrt{D-4},0)\pmod{p}$ and let $\gamma=\alpha^{-1}(s_ys_z)^p\alpha$.
    \end{enumerate}
\end{remark}
\fi
\subsection{If $D$ is zero modulo $p^2$} 
Let us suppose $\left(\frac{D-4}{p}\right)\neq 1$ and $D\equiv 0\pmod{p^2}$. In that case we have $\left(\frac{-4}{p}\right)=(-1)^{(p-1)/2}\neq 1$, so $p\equiv 3\pmod{4}$.

\begin{proposition}
    \label{lem:counting-mod-p-elements}
    Suppose $p\equiv 3\pmod{4}$ and $D\equiv 0\pmod{p}$. Then $|X_D^\ast(\Z/p\Z)|=p(p-3)$.
\end{proposition}
\begin{proof}
    Recall the field $\F_p$ of order $p$, which is isomorphic to $\Z/p\Z$. Its quadratic extension is the unique field $\F_{p^2}$ of order $p^2$. Now the multiplicative group of invertible elements $\F_{p^2}^\times$ of $\F_{p^2}$ is generated by a primitive $(p^2-1)$-th roots of unity, denoted by $\zeta_{p^2-1}$. We denote by $\zeta_{p-1}=\zeta_{p^2-1}^{p+1}$ and $\zeta_{p+1}=\zeta_{p^2-1}^{p-1}$ respectively; notice that $\zeta_{p-1}$ generates $\F_p^\times$.

    Because $p\equiv 3\pmod{4}$, for any point $(x,y,z)\in X_0^\ast(\F_p)$ we have $x\neq 0,\pm 2$. Otherwise, (i) that $x=0$ yields either $y/z=\sqrt{-1}$, a contradiction to $\sqrt{-1}\notin\F_p$, or $(x,y,z)=(0,0,0)$, a point excluded in $X_0^\ast$, and (ii) that $x=\pm 2$ yields $\frac12(y\mp z)=\sqrt{-1}$, again a contradiction.

    For each point $(x,y,z)\in X_0^\ast(\F_p)$, we can write $x=\zeta+\zeta^{-1}$ by some $\zeta\in\F_{p^2}^\times$.
    \begin{itemize}
        \item If $\zeta\in\F_p$, then we can uniquely write $\zeta=\zeta_{p-1}^k$ by some $0<k<\frac12(p-1)$ (the range of $k$ is designed to exclude $x=\pm 2$).
    \end{itemize}
    We call such $x$ \emph{hyperbolic}, following terminologies in \cite[\S{2.1}]{BGS16details}. There are $\frac12(p-1)-1=\frac12(p-3)$ many hyperbolic $x$'s.
    \begin{itemize}
        \item If $\zeta\notin\F_p$, then we can uniquely write $\zeta=\zeta_{p+1}^\ell$ with $0<\ell<\frac12(p+1)$, $\ell\neq\frac14(p+1)$ (the range of $\ell$ is designed to exclude $x=0,\pm 2$).
    \end{itemize}
    We call such $x$ \emph{elliptic}, following loc. cit.. There are $\frac12(p+1)-2=\frac12(p-3)$ many nonzero elliptic $x$'s.

    For a hyperbolic $x$, we parametrize other points $(y,z)$ as follows. From $x=\zeta+\zeta^{-1}$, we may turn the equation $x^2+y^2+z^2=xyz$ to
    \[(y-\zeta z)(y-\zeta^{-1}z)=-x^2.\]
    As $y,z,\zeta\in\F_p$ and $\zeta\neq\zeta^{-1}$, we see that $(y,z)$ is uniquely determined by $\lambda\in\F_p^\times$ thanks to the equations
    \begin{align*}
        y-\zeta z&=\lambda, \\
        y-\zeta^{-1}z&=-x^2\lambda^{-1}.
    \end{align*}
    So for each hyperbolic $x$, there are $(p-1)$ many points of the same $x$-coordinate.

    For a nonzero elliptic $x$, we parametrize other points $(y,z)$ by
    \[(y-\zeta z)(y-\zeta^{-1}z)=-x^2,\]
    but this time, as $\zeta\in\F_{p^2}\setminus\F_p$, we have $y-\zeta z\in\F_{p^2}\setminus\F_p$. The number $\lambda=y-\zeta z$ is nonzero and its Galois conjugate is $\lambda^p=y-\zeta^{-1}z=-x^2\lambda^{-1}$, so in particular $\lambda^{p+1}=-x^2$ follows. Now put $-x^2=\zeta_{p-1}^n$ and $\lambda=\zeta_{p^2-1}^m$. That $\lambda^{p+1}=-x^2$ yields
    \[\frac{(p+1)m}{p^2-1}\equiv\frac{n}{p-1}\pmod{1},\]
    hence $m\equiv n\pmod{p-1}$. There are $(p+1)$ many such residue classes $m$ modulo $(p^2-1)$, so we conclude that there are $(p+1)$ many choices of $\lambda=y-\zeta z$, which will determine $(y,z)$ uniquely. So for each nonzero elliptic $x$, there are $(p+1)$ many points of the same $x$-coordinate.

    Combining the two, we see that
    \[|X_0^\ast(\F_p)|=\frac12(p-3)\cdot(p-1)+\frac12(p-3)\cdot(p+1)=p(p-3). \qedhere\]
\end{proof}

We combine this counting result with the following theorem. This is an easy consequence of D. Martin's proof of modulo $p$ divisibility of $\Gamma$-orbits.

\begin{theorem}[D. Martin]
    \label{thm:martin-chen}
    Let $p\equiv 3\pmod{4}$, $p>3$, and $k\geq 1$. Any $\Gamma$-orbit of $X_0^\ast(\Z/p^k\Z)$ has cardinality divisible by $p^k$.
\end{theorem}
\begin{proof}
    The key is that every point $(x,y,z)\in X_0^\ast(\Z_p)$ has $xyz\not\equiv 0\pmod{p}$ and hence all division operations of the proof in \cite{Martin2025} works, and all modulo $p$ countings generalize to those modulo $p^k$.
\end{proof}

\begin{proposition}
    \label{lem:pigeonhole-mod-pp}
    Suppose $p>3$, $\left(\frac{D-4}{p}\right)\neq 1$, and $D\equiv 0\pmod{p^2}$. For each point $\mathbf{p}\in X_D^\ast(\Z_p)$, there exists $\gamma\in\Gamma$ such that $\dist(\mathbf{p},\gamma.\mathbf{p})=p^{-1}$.
\end{proposition}
\begin{proof}
    We know that $p\equiv 3\pmod{4}$. Let $S=\Gamma.\mathbf{p}$ be the $\Gamma$-orbit. Then the reduction $\red_2(S)\subset X_D^\ast(\Z/p^2\Z)=X_0^\ast(\Z/p^2\Z)$ has size at least $p^2$, by Theorem \ref{thm:martin-chen}. As the further reduction $\red_1(S)\subset X_D^\ast(\Z/p\Z)=X_0^\ast(\F_p)$ has size at most $p(p-3)<p^2$ (by Proposition \ref{lem:counting-mod-p-elements}), there exist $\gamma_1.\mathbf{p},\gamma_2.\mathbf{p}\in S$ that are congruent modulo $p$ but distinguished modulo $p^2$, i.e., $\dist(\gamma_1.\mathbf{p},\gamma_2.\mathbf{p})=p^{-1}$. Take $\gamma=\gamma_1^{-1}\gamma_2$.
\end{proof}

\section{Minimal Level 1 Polydisks}
\label{sec:criteria-of-good-polydisk}

A level 1 polydisk $U$ is a \emph{minimal (level 1) polydisk} if its stabilizer $\Stab_\Gamma(U)$ in $\Gamma$ act minimally on it. In this section, we demonstrate that there exists a minimal polydisk in $X_D^\ast(\Z_p)$ for any $p>3$. By this, we prove our Theorem \ref{thm:main}.

\subsection{Showing minimality}

All the minimality claim of this section can be decomposed into the following steps sketched below.

\begin{lemma} 
    \label{lem:minimality-criterion}
    Let $\Delta$ be a group of analytic automorphisms of $\Z_p^2$ verifying
    \begin{enumerate}[(i)]
        \item \emph{(Residual Transitivity)} the induced action of $\Delta$ on $(\Z_p/p\Z_p)^2$ is transitive, and
        \item \emph{(Minimal Subdisk)} there exists a subdisk $(u_0+p\Z_p)\times(v_0+p\Z_p)\subset\Z_p^2$ and a subgroup $\Delta_0\subset\Delta$ where $\Delta_0$ acts invariantly and minimally on that subdisk.
    \end{enumerate}
    Then $\Delta$ acts minimally on $\Z_p^2$.
\end{lemma}
\begin{proof}
    Let $\mathbf{p},\mathbf{q}\in\Z_p^2$ be any points. We show that for every $\epsilon>0$ there exists $\delta\in\Delta$ such that $\mathsf{dist}(\mathbf{p},\delta(\mathbf{q}))<\epsilon$.

    By (i), we have elements $\alpha,\beta\in\Delta$ such that $\alpha(\mathbf{p})$ and $\beta(\mathbf{q})$ are in $(u_0+p\Z_p)\times(v_0+p\Z_p)$.
    By (ii), we have an element $\gamma\in\Delta$ such that $\mathsf{dist}(\alpha(\mathbf{p}),\gamma\beta(\mathbf{q}))<\epsilon$. It follows that $\mathsf{dist}(\mathbf{p},\alpha^{-1}\gamma\beta(\mathbf{q}))<\epsilon$. Put $\delta=\alpha^{-1}\gamma\beta$.
\end{proof}

\begin{remark}
    Propositions \ref{lem:local-minimality} or \ref{lem:conjugate-local-minimality} will be a key to establish the minimal subdisk step.
\end{remark}

The same logic applies to yield the following trick.
\begin{corollary}
    \label{lem:minimality-criterion-surface}
    Let $\Delta$ be a group of analytic automorphisms of the surface $X_D^\ast(\Z_p)$ verifying
    \begin{enumerate}[(i)]
        \item \emph{(Residual Transitivity)} the induced action of $\Delta$ on $X_D^\ast(\Z_p/p\Z_p)$ is transitive, and
        \item \emph{(Minimal level 1 Subdisk)} there exists a level 1 subdisk $U$ and a subgroup $\Delta_0\subset\Delta$ such that $\Delta_0$ acts invariantly and minimally on $U$.
    \end{enumerate}
    Then $\Delta$ acts minimally on $X_D^\ast(\Z_p)$.
\end{corollary}

\subsection{Existence of a Minimal Polydisk}

By Lemma \ref{lem:minimality-criterion}, we show the existence of a minimal level 1 polydisk in $X_D^\ast(\Z_p)$.

\begin{lemma}
    \label{lem:special-point}
    Let $p\geq 5$, where we let $D\not\equiv 3\pmod{p}$ if $p=5$. If $\left(\frac{D-4}{p}\right)=1$, there exists a point $(x_0,y_0,z_0)\in X_D^\ast(\Z_p)$ such that
    \begin{itemize}
        \item $x_0\equiv\pm 2\pmod{p}$, while $y_0,z_0\not\equiv\pm 2\pmod{p}$, and
        \item all partial derivatives of $P(x,y,z)=x^2+y^2+z^2-xyz$ are nonzero modulo $p$ at the point.
    \end{itemize}
\end{lemma}
\begin{proof}
    Choose a point $(x_0,y_0,z_0)=(2,t+\sqrt{D-4},t)$. For any $t\in\Z_p$, this is a point in $X_D^\ast(\Z_p)$, as $P(x_0,y_0,z_0)=D$, $\partial_yP(x_0,y_0,z_0)=2\sqrt{D-4}\not\equiv 0\pmod{p}$, and $\partial_zP(x_0,y_0,z_0)=-2\sqrt{D-4}\not\equiv 0\pmod{p}$. We furthermore pose the following conditions:
    \begin{itemize}
        \item $(t+\sqrt{D-4})^2\not\equiv 4\pmod{p}$, to have $y_0\not\equiv\pm 2\pmod{p}$;
        \item $t^2\not\equiv 4\pmod{p}$, to have $z_0\not\equiv\pm 2\pmod{p}$; and
        \item $4-t(t+\sqrt{D-4})\not\equiv 0\pmod{p}$, to have $\partial_xP\not\equiv 0\pmod{p}$.
    \end{itemize}
    The conditions exclude at most 6 choice of modulo $p$ residues of $t$, so if $p>6$ then we can find $t\in\Z_p$ satisfying the aboves. If $p=5$ and $D\not\equiv 3\pmod{5}$, then $D\equiv 0\pmod{5}$ must hold and $t=0$ works.
\end{proof}

\begin{proposition}
    \label{lem:minimal-polydisk-eg}
    Suppose $p>3$. Let $D\in\Z_p$ be a parameter such that either $D\equiv 0\pmod{p^2}$ or $\left(\frac{D-4}{p}\right)=1$, where we let $D\not\equiv 3\pmod{5}$ if $p=5$.
    
    Let $(x_0,y_0,z_0)\in X_D^\ast(\Z_p)$ be a point. If $\left(\frac{D-4}{p}\right)=1$, we specifically choose $(x_0,y_0,z_0)$ as in Lemma \ref{lem:special-point}; if otherwise, choose an arbitrary point. 
    Then the level 1 polydisk containing $(x_0,y_0,z_0)$ is minimal.
\end{proposition}
\begin{proof}
    Denote by $\mathbf{p}=(x_0,y_0,z_0)$ the point and by $U$ the level 1 polydisk containing $\mathbf{p}$. Denote by $P(x,y,z)=x^2+y^2+z^2-xyz$. If $\left(\frac{D-4}{p}\right)\neq1$, assume that $\partial_xP(\mathbf{p})=2x_0-y_0z_0\not\equiv 0\pmod{p}$ by shuffling coordinates.

    By an argument in Proposition \ref{lem:counting-mod-p-elements} if $D\equiv 0\pmod{p^2}$ with $\left(\frac{D-4}{p}\right)\neq 1$, or by the hypothesis if $\left(\frac{D-4}{p}\right)=1$, we have $y_0,z_0\not\equiv\pm 2\pmod{p}$. Hence by Lemma \ref{lem:examples-of-stabilizers}, $g=(s_zs_x)^{(p^2-1)/4}$ and $h=(s_xs_y)^{(p^2-1)/4}$ are stabilizing $U$.

    Here, we may assume that $T_p(y_0)=y_0$ and $T_p(z_0)=z_0$. To see this, as $\partial_xP\not\equiv 0\pmod{p}$, start with an analytic map $\xi\colon(y_0+p\Z_p)\times(z_0+p\Z_p)\to(x_0+p\Z_p)$ so that the graph of $\xi$ is $U$. In particular, if we denote by $y_1$ (respt. $z_1$) the fixed point of $T_p$ in the disk $y_0+p\Z_p$ (respt. $z_0+p\Z_p$), then we may change $(x_0,y_0,z_0)$ to $(\xi(y_1,z_1),y_1,z_1)$ yet keep all the required congruences in the hypothesis.

    Next, we apply Proposition \ref{lem:pigeonhole-mod-pp-1} (if $\left(\frac{D-4}{p}\right)=1$) or \ref{lem:pigeonhole-mod-pp} (if otherwise) to find $\gamma\in\Gamma$ such that $\dist(\mathbf{p},\gamma.\mathbf{p})=p^{-1}$.

    Let $\Psi$ be the parametrization of $U$ based on $(y,z)$-coordinates, centered at $(y_0,z_0)$. Recall $g=(s_zs_x)^{(p^2-1)/4}$, $h=(s_xs_y)^{(p^2-1)/4}$, and $\gamma$ in previous paragraphs. We show that $\Delta=\langle g,h,\gamma\rangle$ acts minimally on $U$, or equivalently, $\Psi^{-1}\Delta\Psi$ acts minimally on $\Z_p^2$.

    We show the residual transitivity. Because $T_p(y_0)=y_0$ and $T_p(z_0)=z_0$, \eqref{eqn:expansion-g-mod-p} and \eqref{eqn:expansion-h-mod-p} yield
    \begin{align*}
        \Psi^{-1}g\Psi\begin{pmatrix} u \\ v \end{pmatrix} &\equiv\begin{pmatrix} 1 & 0 \\ c_1 & 1 \end{pmatrix}\begin{pmatrix} u \\ v \end{pmatrix}\pmod{p}, \\
        \Psi^{-1}h\Psi\begin{pmatrix} u \\ v \end{pmatrix} &\equiv\begin{pmatrix} 1 & c_2 \\ 0 & 1 \end{pmatrix}\begin{pmatrix} u \\ v \end{pmatrix}\pmod{p},
    \end{align*}
    where $c_1,c_2$ are invertibe numbers as seen in \eqref{eqn:expansion-g-mod-p} and \eqref{eqn:expansion-h-mod-p}. Observe that the linear parts of $\Psi^{-1}g\Psi$ and $\Psi^{-1}h\Psi$ generate the group $\mathsf{SL}_2(\F_p)$, hence $\Psi^{-1}\langle g,h\rangle\Psi$ has two orbits in $\F_p^2=(\Z_p/p\Z_p)^2$: one is the origin $(0,0)$, and another is the set of all nonzero points.
    
    Let $(x',y',z')=\gamma.(x_0,y_0,z_0)$. Then at least one of $y'$ or $z'$ is different from $y_0$ or $z_0$ (respectively) modulo $p^2$. In particular, $\Psi^{-1}\gamma\Psi$ sends $(0,0)$ to a nonzero point, and therefore $\Psi^{-1}\langle g,h,\gamma\rangle\Psi$ has one orbit in $(\Z_p/p\Z_p)^2$, showing the residual transitivity.

    We show that $(u,v)+(p\Z_p)^2$ is a minimal subdisk by the subgroup $\Delta_0=\langle g^p,h^p\rangle$, if $(u,v)\in\Z_p^2$ is such that $uv\not\equiv 0\pmod{p}$. Observe that we can apply Proposition \ref{lem:local-minimality} to $\Psi^{-1}g^p\Psi$ and $\Psi^{-1}h^p\Psi$, as
    \begin{enumerate}[(a)]
        \item $\Psi^{-1}g^p\Psi\equiv\Psi^{-1}h^p\Psi\equiv\mathrm{id}\pmod{p}$ by \eqref{eqn:expansion-gp-mod-pp} and \eqref{eqn:expansion-hp-mod-pp}, and
        \item \eqref{eqn:indep-vectorial-part} holds: we compute, from \eqref{eqn:expansion-gp-mod-pp} and \eqref{eqn:expansion-hp-mod-pp},
        \begin{align*}
            \frac{\Psi^{-1}g^p\Psi(u,v)-(u,v)}{p} &\equiv \begin{pmatrix} 0 \\ c_1u \end{pmatrix}\pmod{p}, \\
            \frac{\Psi^{-1}h^p\Psi(u,v)-(u,v)}{p} &\equiv \begin{pmatrix} c_2v \\ 0 \end{pmatrix}\pmod{p},
        \end{align*}
        so \eqref{eqn:indep-vectorial-part} is equivalent to $c_1c_2uv\not\equiv 0\pmod{p}$ which was assumed.
    \end{enumerate}
    Hence $\Psi^{-1}\Delta_0\Psi$ acts minimally on $(u,v)+(p\Z_p)^2$.
    By Lemma \ref{lem:minimality-criterion}, we conclude that $\Delta$ acts minimally on $U$.
\end{proof}

Although we have excluded the case $p=5$ and $D\equiv 3\pmod{5}$ in Proposition \ref{lem:minimal-polydisk-eg}, we can secure this case anyways.
\begin{proposition}
    \label{lem:minimal-polydisk-exceptional}
    Let $p=5$ and $D\equiv 3\pmod{5}$. Then the level 1 polydisk $U$ of $X_D^\ast(\Z_5)$ containing $(\sqrt{D-4},2,0)$ is minimal.
\end{proposition}
\begin{proof}
    Note that all partial derivatives of $P(x,y,z)=x^2+y^2+z^2-xyz$ are nonzero modulo $5$ at the point $(\sqrt{D-4},2,0)$, and that $\sqrt{D-4}\equiv\pm 2\pmod{5}$. Let $\Psi$ be the parametrization of $U$ based on $(y,z)$-coordinates, centered at $(2,0)$; note that $2$ and $0$ are fixed points of $T_5$.
    
    By Lemma \ref{lem:examples-of-stabilizers}, $(s_ys_z)^5$, $(s_zs_x)^5$, and $(s_xs_y)^6$ are fixing $U$. By \eqref{eqn:expansion-para-f-mod-p} and \eqref{eqn:expansion-h-mod-p},
    \begin{align*}
        \Psi^{-1}(s_ys_z)^5\Psi\begin{pmatrix} u \\ v \end{pmatrix} &\equiv \begin{pmatrix} u \\ v \end{pmatrix} + \begin{pmatrix} 4 \\ -2\sqrt{D-4} \end{pmatrix}\pmod{p}, \\
        \Psi^{-1}(s_xs_y)^6\Psi\begin{pmatrix} u \\ v \end{pmatrix} &\equiv \begin{pmatrix} 1 & -6\sqrt{D-4} \\ 0 & 1 \end{pmatrix}\begin{pmatrix} u \\ v \end{pmatrix}\pmod{p}.
    \end{align*}
    To estimate $\Psi^{-1}(s_zs_x)^5\Psi$, we follow the proof of Lemma \ref{lem:expansion-g-h}: set $N=10$ and get
    \begin{align*}
        \Psi^{-1}(s_zs_x)^5\Psi\begin{pmatrix} u \\ v \end{pmatrix} &= \begin{pmatrix} u \\ \frac15(U_{10}(2+5u)\cdot 5v-U_9(2+5u)\xi) \end{pmatrix},
        \intertext{which can be directly computed and reduced modulo 5 to}
        &\equiv\begin{pmatrix} u \\ v \end{pmatrix} + \begin{pmatrix} 0 \\ -2\sqrt{D-4} \end{pmatrix}\pmod{5}.
    \end{align*}
    A similar computation yields
    \[\Psi^{-1}(s_zs_x)^{25}\Psi\begin{pmatrix} u \\ v \end{pmatrix}\equiv\begin{pmatrix} u \\ v \end{pmatrix} + 5\begin{pmatrix} 0 \\ -2\sqrt{D-4}\end{pmatrix}\pmod{25}.\]
    
    As $\Psi^{-1}(s_ys_z)^5\Psi$ and $\Psi^{-1}(s_zs_x)^5\Psi$ are linearly independent translations modulo $5$, they establish the residual transitivity. Moreover, $(0,0)+(5\Z_5)^2$ is a minimal subdisk thanks to Proposition \ref{lem:conjugate-local-minimality}: one can apply $f=\Psi^{-1}(s_zs_x)^{25}\Psi$ and $g=\Psi^{-1}(s_xs_y)^6\Psi$ to the proposition and get the minimality. This shows that $U$ is minimal.
\end{proof}

\subsection{Proof of Theorem \ref{thm:main}}
\label{sec:proof}

Set $\Delta=\Aut(X_D^\ast)$. We can clear the existence of minimal level 1 subdisk, required in Corollary \ref{lem:minimality-criterion-surface}, by Propositions \ref{lem:minimal-polydisk-eg} and \ref{lem:minimal-polydisk-exceptional}. As we have assumed residual transitivity, the minimality is shown.

\section{Generalizations}
\label{sec:generalization}

For parameters other than $D\equiv 0\pmod{p^2}$ or $\left(\frac{D-4}{p}\right)=1$, there might be a chance that residual transitivity does not imply minimality. This section discusses about when such cases may happen.

\subsection{Issues regarding ``lack of small moves''}

If $D\not\equiv 0\pmod{p^2}$ and $\left(\frac{D-4}{p}\right)\neq 1$, we see that it is not easy to generaize Propositions \ref{lem:pigeonhole-mod-pp-1} and \ref{lem:pigeonhole-mod-pp}: we may fail to find $\mathbf{p}\in X_D^\ast(\Z_p)$ and $\gamma\in\Gamma$ such that $\dist(\mathbf{p},\gamma.\mathbf{p})=p^{-1}$.

One easy guess to try is, assuming $\partial_xP=2x_0-y_0z_0\in\Z_p^\times(=\Z_p\setminus p\Z_p)$, to put $\gamma$ an appropriate power of $(s_ys_z)$. But if $\left(\frac{D-4}{p}\right)=-1$, then $x_0\not\equiv\pm 2\pmod{p}$, and $(s_ys_z)^p$ no longer works like Proposition \ref{lem:pigeonhole-mod-pp-1}. It turns out that powers of $(s_ys_z)$ ``keeps an appropriate distance from the the identity'' only when $x_0$ is appropriately distanced from the trace of rational rotation in $x_0+p\Z_p$.
\begin{lemma}
    \label{lem:expansion-nonpara-f}
    Suppose $x_0\not\equiv\pm 2\pmod{p}$ and $\partial_xP(x_0,y_0,z_0)\in\Z_p^\times$, where $P(x,y,z)=x^2+y^2+z^2-xyz$. Denote by $x_1$ the fixed point of $T_p$ on the disk $x_0+p\Z_p$, and by $\Psi$ the parametrization of the level 1 polydisk of $(x_0,y_0,z_0)$, based on $(y,z)$-coordinates and centered at $(y_0,z_0)$. 
    Let $N=(p^2-1)/2$, and
    \begin{align*}
        \mathbf{w}_1 &= \frac{N}{x_0^2-4}\begin{pmatrix}-\partial_zP \\ \partial_yP \end{pmatrix}, & \mathbf{w}_2 &= \frac{-1}{\partial_xP}\begin{pmatrix}\partial_yP \\ \partial_zP \end{pmatrix}.
    \end{align*}
    Then by $f=\Psi^{-1}(s_ys_z)^{N/2}\Psi$ we have
    \begin{align}
        f\begin{pmatrix} u \\ v \end{pmatrix} &\equiv (I_2+\mathbf{w}_1\mathbf{w}_2^\top)\begin{pmatrix} u \\ v \end{pmatrix}+\frac{x_0-x_1}{p}\mathbf{w}_1\pmod{p}. \label{eqn:expansion-npara-f-mod-p}
        %\\ f^p\begin{pmatrix} u \\ v \end{pmatrix} &\equiv (I_2+p\mathbf{w}_1\mathbf{w}_2^\top)\begin{pmatrix} u \\ v \end{pmatrix}+(x_0-x_1)\mathbf{w}_1\pmod{p^2}. \label{eqn:expansion-npara-fp-mod-pp}
    \end{align}
    Here, $A^\top$ denotes the transpose of the matrix $A$.
\end{lemma}
\begin{proof}
    Same analysis as Lemma \ref{lem:expansion-para-f}, except that we now take $\xi\equiv x_0+p\mathbf{w}_2^\top(\begin{smallmatrix} u \\ v \end{smallmatrix})\pmod{p^2}$ into account.
\end{proof}

To make this expansion useful, we need an extra data (here, $T_p(\mathbf{q})$ means $(T_p(x_0),T_p(y_0),T_p(z_0))$ if $\mathbf{q}=(x_0,y_0,z_0)$):
\begin{itemize}
    \item[\XD] There exist a point $\mathbf{p}\in X_D^\ast(\Z_p)$ and an element $\alpha\in\Gamma$ with $P(T_p(\alpha.\mathbf{p}))\not\equiv D\pmod{p^2}$.
\end{itemize}
If we have \XD{} and $\left(\frac{D-4}{p}\right)=-1$, then we can select $\gamma$ from the set
\[\{\alpha^{-1}f^{(p^2-1)/4}\alpha : f=s_ys_z,s_zs_x,\text{ or }s_xs_y\}\]
so that $\dist(\mathbf{p},\gamma.\mathbf{p})=p^{-1}$. Having this, the proof of Proposition \ref{lem:minimal-polydisk-eg} applies to show that $\Aut(X_D^\ast)$ acts minimally on the closed $p^{-1}$-neighborhood of the orbit of $\mathbf{p}$.

By Lemma \ref{lem:Tp-contraction}, if $\mathbf{p}\equiv\mathbf{q}\pmod{p}$ then $T_p(\mathbf{p})\equiv T_p(\mathbf{q})\pmod{p^2}$. It follows that, if \XD{} fails, then the $T_p$-images of each point of $X_D^\ast(\Z_p)$ form an $\Aut(X_D^\ast)$-orbit in $X_D^\ast(\Z/p^2\Z)$. This orbit must be a proper subset of $X_D^\ast(\Z/p^2\Z)$, so the minimality fails. Hence \XD{} is not only a technical assumption but also a pivoting assumption to see if residual transitivity implies minimality.

It is known that, if $\left(\frac{D-4}{p}\right)=-1$, then $|X_D^\ast(\Z/p\Z)|\leq(p-1)^2$. So to have \XD{}, it suffices to see that every $\Aut(X_D^\ast)$-orbit of $X_D^\ast(\Z/p^2\Z)$ has cardinality a multiple of $p^2$. It turns out that, if $X_D^\ast(\Z_p)$ stays close to a finite orbit, then this divisibility is in general false.

\subsection{Issues regarding the Periodic Orbits}

According to \cite{LisovyyTykhyy2014}, Lem. 39 of p. 147, Thm. 1 of p. 149, and Table 4 of pp. 150--151, there are a few exceptional cases where $X_D^\ast$ carries periodic orbits.
\begin{itemize}
    \item If $D$ is a nonzero quadratic residue, then $(0,0,\pm\sqrt{D})$, $(0,\pm\sqrt{D},0)$, and $(\pm\sqrt{D},0,0)$ form a finite orbit.
    \item If $D=4$, any orbit outside of the `cage' (a term from \cite[\S{3.2}]{BGS16details}) forms a small finite orbit.
    \item If $D=2$, then the orbit of $(1,1,1)$ forms a finite orbit, of size 16.
    \item If $\sqrt{2}\in\Z_p$ and $D=3$, then the orbit of $(1,\sqrt{2},\sqrt{2})$ forms a finite orbit, of size 12.
    \item If $\sqrt{5}\in\Z_p$ and $D=3$ or $D=(5\pm\sqrt{5})/2$, then the orbit of $(0,\phi^{-1},\phi)$, $(\phi,\phi,\phi)$, or $(-\phi^{-1},-\phi^{-1},-\phi^{-1})$ forms a finite orbit, of size 72, 40, and 40 respectively, where $\phi=\frac12(1+\sqrt{5})$ is the golden ratio.
\end{itemize}

If $D$ is congruent to one of the above cases modulo $p$ and $p\geq 11$, then modulo $p$ images of these finite orbits are too small to fill up the entire $X_D^\ast(\F_p)$. Hence the residual transitivity fails. 
If $D$ is congruent to one the above cases modulo $p^2$, then the $p^{-2}$-neighborhood of these finite orbits form an invariant set in $X_D^\ast(\Z/p^2\Z)$ of cardinality at most $|X_D^\ast(\Z/p\Z)|$. Hence the minimality fails.

We conjecture that, outside of the $p^{-1}$-neighborhoods of these periodic points, we still have minimality over $p$-adic integer points. We propose the following intermediate steps to the conjecture as follows.
\begin{question}
    \label{lem:question-XD}
    \begin{enumerate}
        \item Suppose $D$ is not congruent modulo $p$ to any of the parameters admitting periodic orbits. %has $\left(\frac{D}{p}\right)\neq 1$, $D\not\equiv 2,4\pmod{p}$, $D\not\equiv 3\pmod{p}$ if $\sqrt{2}\in\Z_p$, and $D\not\equiv 3,(5\pm\sqrt{5})/2\pmod{p}$ if $\sqrt{5}\in\Z_p$. 
        Is every $\Aut(X_D^\ast)$-orbit of $X_D^\ast(\Z/p^2\Z)$ has cardinality a multiple of $p^2$?
        \item Suppose we have a point $\mathbf{p}\in X_D^\ast(\Z_p)$ such that for every periodic point $\mathbf{q}$ listed above, we have $\dist(\gamma.\mathbf{p},\mathbf{q})=1$ for all $\gamma\in\Gamma$. Does \XD{} hold by this point $\mathbf{p}$?
    \end{enumerate}
\end{question}

\subsection*{Acknowledgements} The author would like to thank to Alonso Beaumont Llona, Serge Cantat, Elena Fuchs, Alexander Gamburd, Daniel Martin, and Peter Sarnak for the discussions, comments, and helps regarding this work. The research activities of the author is partially funded by the European Research Council (ERC GOAT 101053021). %The author also want to thank She Yang and Zev Chonoles for motivating Wall--Sun--Sun primes in this study.

\bibliography{references}{}

\newcommand{\etalchar}[1]{$^{#1}$}
\providecommand{\bysame}{\leavevmode\hbox to3em{\hrulefill}\thinspace}
\providecommand{\MR}{\relax\ifhmode\unskip\space\fi MR }
\providecommand{\MRhref}[2]{%
  \href{http://www.ams.org/mathscinet-getitem?mr=#1}{#2}
}
\providecommand{\href}[2]{#2}
\begin{thebibliography}{EFL{\etalchar{+}}23}

\bibitem[Bel06]{Bell2006}
Jason~P. Bell, \emph{A generalised {S}kolem-{M}ahler-{L}ech theorem for affine
  varieties}, J. London Math. Soc. (2) \textbf{73} (2006), no.~2, 367--379.
  \MR{2225492}

\bibitem[Bel08]{Bell2006:corrigendum}
\bysame, \emph{Corrigendum: ``{A} generalised {S}kolem-{M}ahler-{L}ech theorem
  for affine varieties'' [{J}. {L}ondon {M}ath. {S}oc. (2) {\bf 73} (2006), no.
  2, 367--379; mr2225492]}, J. Lond. Math. Soc. (2) \textbf{78} (2008), no.~1,
  267--272. \MR{2427064}

\bibitem[BGS16a]{BGS16details}
Jean Bourgain, Alexander Gamburd, and Peter Sarnak, \emph{Markoff surfaces and
  strong approximation: 1}, 2016.

\bibitem[BGS16b]{BGS16}
Jean Bourgain, Alexander Gamburd, and Peter Sarnak, \emph{Markoff triples and
  strong approximation}, C. R. Math. Acad. Sci. Paris \textbf{354} (2016),
  no.~2, 131--135. \MR{3456887}

\bibitem[Che24]{Chen2024}
William~Y. Chen, \emph{Nonabelian level structures, {N}ielsen equivalence, and
  {M}arkoff triples}, Ann. of Math. (2) \textbf{199} (2024), no.~1, 301--443.
  \MR{4681147}

\bibitem[CJ24]{CJ24}
Serge Cantat and Seung~uk Jang, \emph{Orbits of automorphism groups of affine
  surfaces over $p$-adic fields}, Preprint, {arXiv}:2410.08579 [math.{AG}]
  (2024), 2024.

\bibitem[CX18]{CantatXie2018}
Serge Cantat and Junyi Xie, \emph{Algebraic actions of discrete groups: the
  {$p$}-adic method}, Acta Math. \textbf{220} (2018), no.~2, 239--295.
  \MR{3849285}

\bibitem[EFL{\etalchar{+}}23]{EFLMT2023}
Jillian Eddy, Elena Fuchs, Matthew Litman, Daniel Martin, and Nico Tripeny,
  \emph{Connectivity of {M}arkoff mod-p graphs and maximal divisors}, 2023.

\bibitem[EH74]{el-huti}
M.~H. \`El'-Huti, \emph{Cubic surfaces of {M}arkov type}, Mat. Sb. (N.S.)
  \textbf{93(135)} (1974), 331--346, 487. \MR{342518}

\bibitem[Gam23]{Gamburd2022}
Alexander Gamburd, \emph{Arithmetic and dynamics on varieties of {M}arkoff
  type}, I{CM}---{I}nternational {C}ongress of {M}athematicians. {V}ol. 3.
  {S}ections 1--4, EMS Press, Berlin, [2023] \copyright 2023, pp.~1800--1836.
  \MR{4680301}

\bibitem[Gol03]{Goldman2003}
William~M. Goldman, \emph{The modular group action on real {${\rm
  SL}(2)$}-characters of a one-holed torus}, Geom. Topol. \textbf{7} (2003),
  443--486. \MR{2026539}

\bibitem[Hor02]{Horadam2002}
A.~F. Horadam, \emph{Vieta polynomials}, vol.~40, 2002, A special tribute to
  Calvin T. Long, pp.~223--232. \MR{1913347}

\bibitem[LT14]{LisovyyTykhyy2014}
Oleg Lisovyy and Yuriy Tykhyy, \emph{Algebraic solutions of the sixth
  {P}ainlev\'e{} equation}, J. Geom. Phys. \textbf{85} (2014), 124--163.
  \MR{3253555}

\bibitem[Mar25]{Martin2025}
Daniel~E. Martin, \emph{A new proof of {C}hen's theorem for {M}arkoff graphs},
  2025.

\bibitem[Poo14]{Poonen2014}
Bjorn Poonen, \emph{{$p$}-adic interpolation of iterates}, Bull. Lond. Math.
  Soc. \textbf{46} (2014), no.~3, 525--527. \MR{3210707}

\end{thebibliography}
\bibliographystyle{amsalpha}

\end{document}